\documentclass[a4paper,11pt]{article}%

\newcounter{dummy}
\usepackage{dina42e}
\usepackage{titlesec}
\usepackage{amsmath, amsthm, amsfonts, amssymb, mathrsfs}%
\allowdisplaybreaks
\usepackage{setspace}
\usepackage{lmodern}
\usepackage[color links = true, linkcolor=black, citecolor=black]{hyperref}
\usepackage{booktabs}
\usepackage{enumitem} 
\usepackage{scalerel}
\usepackage[usestackEOL]{stackengine}
\usepackage{color}

\let\oldmathchoice\mathchoice

\def\dashint{\let\mathchoice\oldmathchoice\,\ThisStyle{\ensurestackMath{%
			\stackinset{c}{.2\LMpt}{c}{.5\LMpt}{\SavedStyle-}{%
				\SavedStyle\phantom{\int}}}%
		\setbox0=\hbox{$\SavedStyle\int\,$}\kern-\wd0}\int%
	\let\mathchoice\newmathchoice}

\newcommand{\R}{\mathbb{R}}
\newcommand{\N}{\mathbb{N}}

\newcommand{\wcon}{\rightharpoonup}
\newcommand{\e}{\varepsilon}

\newcommand{\Rtimes}{\mathbb{R}^{3\times 3}}
\newcommand{\Rntimesn}{\mathbb{R}^{n\times n}}
\newcommand{\Colonh}{:_h}
\newcommand{\Sym}{\operatorname{sym}}
\newcommand{\Skew}{\operatorname{skew}}
\newcommand{\divh}{\operatorname{div}_h}

\newcommand{\tr}[1]{\operatorname{tr}_{#1}}
\newcommand{\tw}{\tilde{w}}
\newcommand{\dist}{\operatorname{dist}}

\newcommand{\sym}{\operatorname{sym}}

\renewcommand{\L}{\mathcal{L}}

\makeatletter
\newcommand{\myitem}[1][]{\item[#1]\refstepcounter{dummy}\def\@currentlabel{#1}}
\makeatother

\newtheorem{theorem}{Theorem}[section]
\newtheorem{lem}[theorem]{Lemma}
\newtheorem{Coro}[theorem]{Corollary}

\theoremstyle{definition}

\newtheorem{bem}[theorem]{Remark}

\numberwithin{equation}{section}

\ifx\cs\documentclass \footheight 12pt \fi \footskip 30pt 
\titleformat{\section}{\normalfont\Large\bfseries}{\thesection}{1em}{}
\titleformat{\subsection}{\normalfont\large\bfseries}{\thesubsection}{1em}{}
\titleformat{\subsubsection}{\normalfont\normalsize\bfseries}{\thesubsubsection}{1em}{}
\titleformat{\paragraph}[runin]{\normalfont\normalsize\bfseries}{\theparagraph}{1em}{}
\titleformat{\subparagraph}[runin]{\normalfont\normalsize\bfseries}{\thesubparagraph}{1em}{}

\begin{document}
	\begin{titlepage}
		\title{Large Times Existence for Thin Vibrating Rods}
		\author{Helmut Abels\thanks{Fakult\"{a}t f\"{u}r Mathematik, Universit\"{a}t Regensburg, 93040 Regensburg, Germany, e-mail:
				helmut.abels@mathematik.uni-regensburg.de} \space and Tobias Ameismeier\thanks{Fakult\"{a}t f\"{u}r Mathematik, Universit\"{a}t Regensburg, 93040 Regensburg, Germany, e-mail:
				tobias.ameismeier@mathematik.uni-regensburg.de}}
	\end{titlepage}
	\maketitle
	\begin{abstract}
		We consider the dynamical evolution of a thin rod described by an appropriately scaled wave equation of nonlinear elasticity. Under the assumption of well-prepared initial data and external forces, we prove that a solution exists for arbitrarily large times, if the diameter of the cross section is chosen sufficiently small. The scaling regime is such that the limiting equations are linear. 
	\end{abstract}
	
	\vspace{2mm}	
	
	\noindent{\textbf{AMS-Classification:}} Primary: 74B20, 
	Secondary:
	35L20, 
	35L70,  
	74K10 
	\section{Introduction}
	The starting point of the analysis of thin vibrating rods in this contribution is the nonlinear elastic energy
	\begin{equation*}
		\tilde{\mathcal{E}}^h(z) := \frac{1}{h^2} \int_{\Omega_h} W(\nabla z(x)) - f^h(x)\cdot(z(x) - x) dx,
	\end{equation*}
	where $z\in W^1_2 (\Omega_h; \R^3)$ is a deformation, $S\subset\R^2$ denotes the cross section of the rod, $\Omega_h := (0,L)\times hS$ for some $L>0$ is the scaled reference configuration of the rod and $f^h$ describes an external volume load. For convenience we rescale the energy to an $h$-independent reference domain $\Omega := \Omega_1$ and obtain
	\begin{equation*}
		\mathcal{E}^h(y) := \int_\Omega W(\nabla_h y(x)) - f^h(x) \cdot (y(x) - x^h) dx,
	\end{equation*}
	with $\nabla_h := (\partial_{x_1}, \frac{1}{h}\partial_{x_2}, \frac{1}{h}\partial_{x_3})^T$, $x^h :=(x_1, hx_2, hx_3)$ and $y\in W^1_2(\Omega; \R^3)$. The limit of this energy depends on the scaling property of $f^h$ and thus on the scaling of $\mathcal{E}^h$ with respect to $h \to 0$. An in depth analysis of the convergence properties in the sense of $\Gamma$-convergence can be found in \cite{MoraMueller2003, MoraMueller2004} and in case of a curved reference configuration see for instance \cite{Scardia2006, Scardia2009}. In case of periodic boundary conditions we refer to \cite{AmeismeierDiss}. 
	
	In detail energies of order $h^{2\alpha - 2}$ for $\alpha\geq 3$ correspond to $f^{h}$ being of order $h^\alpha$. The choice of $\alpha = 3$ and $\alpha>3$ leads to a von Kármán limiting energy and a linearised theory, respectively. Deformations of this scaling behaviour are close to a rigid motion. The limit energies for $\frac{1}{h^{2\alpha-2}}\mathcal{E}^{h}$ are derived as
	\begin{equation*}
	\mathcal{E}_\alpha(u,v_2,v_3,w,R') := \mathcal{I}_\alpha(u,v_2,v_3,w) - \int_0^L (R')^T \begin{pmatrix}
	f_2 \\ f_3
	\end{pmatrix} \cdot \begin{pmatrix}
	v_2\\v_3
	\end{pmatrix} dx_1,
	\end{equation*}
	where
	\begin{equation*}
	\mathcal{I}_{\alpha} (u,v_2,v_3,w) := \begin{cases}
	\displaystyle\frac{1}{2} \int_0^L Q^0(u_{,1} + \tfrac{1}{2}(v_{2,1}^2 + v_{3,1}^2), A_{,1}) dx_1, \qquad&\text{if } \alpha = 3,\\
	\displaystyle\frac{1}{2} \int_0^L Q^0(u_{,1}, A_{,1})dx_1, \qquad&\text{if } \alpha > 3.
	\end{cases}
	\end{equation*}
	Here $u, w\in H^1_{per}(0,L)$, $v\in H^2_{per}(0,L;\R^2)$ are the limits of appropriately scaled means of $y^{h}$ and $R'$ is the $2\times 2$-lower left submatrix of the limit of an approximating rotation.
	The matrix $A\in H^1_{per}(0,L; \Rtimes)$ is given by
	\begin{equation*}
	A = \begin{pmatrix}
	0 & -v_{2,1} & -v_{3,1}\\
	v_{2,1} & 0 & -w\\
	v_{3,1} & w & 0
	\end{pmatrix}.
	\end{equation*}
	Moreover, $Q^0\colon \R\times \Rtimes_{skew} \to [0,\infty)$ is defined by
	\begin{equation*}
	Q^0(t,F) := \min_{\varphi\in H^1(S,\R^3)} \int_S Q_3\Big(te_1 + F\begin{pmatrix}
	0 \\ x'
	\end{pmatrix}\Big| \varphi_{,2} \Big| \varphi_{,3}\Big) dx'
	\end{equation*}
	with $Q_3(G) := D^2 W(Id)[G,G]$, the quadratic form of linearised elasticity.
	
	In our work we study the case $\alpha\geq 4$ in the dynamic situation. The basic equations from continuum mechanics emerge from the balance of linear and angular momentum and the balance of energy. Formally the evolution preserves the total energy
	\begin{equation*}
		\int_\Omega \bigg(\frac{|\partial_t y|^2}{2} + W(\nabla_h y) - f^h\cdot y \bigg)dx
	\end{equation*}
	if $f^h$ is independent of $t$ and the Piola-Kirchhoff stress satisfies appropriate boundary conditions.
		
	According to the scaling behaviour of the $S$-means of $y^h$  we expect
	\begin{align*}
		y_1 - x_1 \sim h^{\alpha-1},\quad \begin{pmatrix} y_2\ \\ y_3 \end{pmatrix} \sim h^{\alpha-2} \qquad\text{ for } \alpha\geq 4.
	\end{align*}
	Moreover, the $\Gamma$-convergence results suggest $f^h\sim h^{\alpha}$. For simplicity we choose $f_1 \equiv 0$. For some results with normal and tangential forces we refer to \cite{MuellerLecumberry}. In order to balance the kinetic and elastic part of the total energy we rescale the time via $\tau = h t$. Hence, we obtain
	\begin{equation*}
		\mathcal{E}^h_{tot}(y) = h^{2\alpha - 2} \int_\Omega \frac{1}{h^{2\alpha - 4}}\frac{|\partial_\tau y|^2}{2} + \frac{1}{h^{2\alpha - 2}} W(\nabla_h y) - \frac{1}{h^\alpha} f^h \cdot \frac{1}{h^{\alpha-2}} y dx.
	\end{equation*}
	This leads to the scaled evolution equation with $g^h:= \frac{1}{h^\alpha} f^h$
	\begin{equation}
		\partial_\tau^2 y - \frac{1}{h^2} \divh(DW(\nabla_h y)) = h^{\alpha - 2} g^h\quad\text{ in }\Omega\times[0,T)
		\label{Equation_NonlinearEquationIntroduction}
	\end{equation}
	where $g^h \sim 1$ for $h\to 0$. Furthermore, we assume homogeneous Neumann boundary conditions on the outer surface, periodicity on the end faces of $\Omega$ and suitable initial conditions. For well-prepared initial data we are able to show that for any $T>0$ there exists an $h_0>0$ such that strong solutions exist on $(0,T)$ for all $h\in (0,h_0]$.
	
	In a first step we use the methods of \cite{Koch} in order to establish short time existence for $h>0$, where the time of existence $T>0$ depends on $h$. As the results are proved for a sufficiently smooth bounded domains, we can not directly apply the theorem. The periodic boundary conditions on the end faces of the rod and homogeneous Neumann boundary conditions on the outer surface however admit applying many arguments of the proofs. More details are given in Subsection \ref{subsection::MainResult} and the Appendix \ref{Appendix::ExistenceOfClassicalSolutionsForFixedH}. 
	
	The main difficulty is therefore to prove the existence of a uniform lower bound for $T$ with respect to $h\to 0$. As the global strategy is quite similar to the one used in \cite{AbelsMoraMueller}, we want to give a short overview on the main new difficulties of this work. Generally we use an energy method in order to obtain enough regularity bounds for the solution. With this we can then conclude for well-prepared initial data that the solution has to exist for some short time interval. Using a proof towards contradiction we obtain that for an arbitrary $T>0$ there exists some $h_0>0$ such that for all $h\in (0, h_0]$ the solution exists on $[0, T]$.
	
	In order to do so we use the classical energy ansatz and differentiate \eqref{Equation_NonlinearEquationIntroduction} with respect to $z=t,x_1$ and test with $\partial_t\partial_z y$ to obtain
	\begin{equation*}
		\frac{d}{dt}\frac{1}{2}\Big( \|\partial_t \partial_z y\|^2_{L^2(\Omega)} + \frac{1}{h^2}\Big(D^2 W(\nabla_h y)\partial_z \nabla_h y, \partial_z \nabla_h y\Big)_{L^2(\Omega)}\Big) = (\partial_z f, \partial_t \partial_{(t,x)} y)_{L^2(\Omega)} + \mathcal{R}
	\end{equation*}
	where $\mathcal{R}$ is some remainder term. We use the properties of $D^2 W$ to apply a Gronwall and absorption argument for the remainder. More precisely the following bound is crucial
	\begin{equation*}
		\bigg|\frac{1}{h^2}\Big(D^2 W(\nabla_h y) \nabla_h u, \nabla_h u\Big)_{L^2(\Omega)}\bigg| \geq C\bigg\|\frac{1}{h} \e_h(u)\bigg\|^2_{L^2(\Omega)}.
	\end{equation*}
	with $\e_h(u) :=\sym(\nabla_h u)$. Unfortunately, the boundary conditions do not prevent large rotations around the $x_1$-axis. Therefore the symmetric gradient can not bound the full gradient, which is reflected in an adapted Korn inequality for thin rods
	\begin{equation*}
		\|\nabla_h u\|_{L^2(\Omega)}\leq \frac{C_K}{h}\bigg( \|\e_h(u)\|_{L^2(\Omega)} + \bigg| \int_\Omega u \cdot x^\perp dx\bigg|\bigg)
	\end{equation*}
	where $x^\perp = (0, -x_3, x_2)^T$ and $C_K>0$ independent of $h>0$. We overcome this problem by using the balance of angular momentum for the solution to the non-linear equation.
	
	On the technical level the most difficulties arise from the fact that the cross section of the rod is two dimensional. As a result we need higher regularity bounds in order to obtain $\nabla_h^2 y\in L^\infty$. Due to the boundary conditions on the end faces of the rod, we can easily derive the equation in $x_1$ direction and in time. To obtain bounds on the second scaled gradient in $(x_2, x_3)$ direction we derive a lower dimensional system of the from
	\begin{equation*}
		\left\{
		\begin{aligned}
		-\frac{1}{h^2} \operatorname{div}_{x'}\big(D^2 W(Id)^\approx \nabla_{x'} \varphi(x_1, \cdot)\big) &= \tilde{g}(x_1, \cdot) &&\quad\text{ in } S,\\
		\frac{1}{h}\big(D^2 W(Id)^\approx \nabla_{x'} \varphi(x_1, \cdot)\big) \nu_{\partial S}\bigg|_{\partial S} &= g_N(x_1, \cdot) - a_N(x_1, \cdot)&&\quad\text{ on } \partial S
		\end{aligned}
		\right.
	\end{equation*}
	as we can not algebraically solve for the higher order terms. Here $D^2 W(Id)^\approx$ is a suitable restriction for a two dimensional system and $x_1\in [0,L]$ arbitrary.
%
%

	For a more general introduction to elasticity theory we refer to \cite{Antman} and for a different approach on the behaviour of thin rods using unfolding methods see \cite{Griso}. Furthermore we want to mention the results on convergence of minimizers in the dynamical setting for plates \cite{AbelsMoraMuellerConvergence} and shells \cite{QinYao}, where in the latter the large times existence is still an open problem to the best of the authors' knowledge.
	
	The structure of this contribution is as follows: In Section \ref{section::PreliminariesAndAuxiliaryResults} we introduce the notation and some needed auxiliary results. Most important some specific bounds for the strain energy density $W$ and Korn's inequality in thin rods. Section \ref{section::LargeTimeExistence} is devoted to the main result, which is stated in Section \ref{subsection::MainResult} and proven in Section \ref{subsec::ProofOfMainTheorem}. The analysis of the linearised system is done in Section \ref{subsec::UniformEstimatesForLinearisedSystem}. In the appendix we discuss the existence of classical solutions for fixed $h>0$.

         \smallskip

	\noindent{\textbf{Acknowledgements:} Tobias Ameismeier was supported by the RTG 2339 “Interfaces, Complex Structures,
		and Singular Limits” of the German Science Foundation (DFG). The support is gratefully acknowledged.}

\section{Preliminaries and Auxiliary Results}
\label{section::PreliminariesAndAuxiliaryResults}
\subsection{Notation}
\label{subsection::Notation}
The natural numbers without zero are denoted by $\N$ and $\N_0 := \N \cup\{0\}$. For any $n\in\N$ we denote the norm on $\R^n$, $\Rntimesn$ and the absolute value by $|.|$. With $L^p(M)$, $W^k_p(M)$ and $H^k(M) := W^k_2(M)$, for $p,k\in \N$, we denote the classical Lebesgue and Sobolev spaces for some bounded, open set $M\subset\R^n$. Throughout the paper $\mathcal{L}^n(V)$, $n\in\N$, denotes the space of all $n$-linear mappings $G\colon V^n \to\R$ for a vector space $V$. As common we will use the classical identification of $\mathcal{L}^1(\Rntimesn) = (\Rntimesn)'$ with $\Rntimesn$, i.e., $G\in \mathcal{L}^1(\Rntimesn)$ is identified with $A\in\Rntimesn$ such that
\begin{equation*}
G(X) = A : X \quad\text{ for all } X\in\Rntimesn
\end{equation*}
where $A:X = \sum_{i,j=1}^n a_{ij} x_{ij}$ is the standard inner product on $\Rntimesn$. Similarly, $G\in \mathcal{L}^2(\Rntimesn)$ is identified with $\tilde{G}\colon\Rntimesn\to\Rntimesn$ defined by
\begin{equation}
\tilde{G}X : Y = G(X,Y)\quad \text{ for all } X, Y\in\Rntimesn.
\label{IdentificationL2withTensor4Order}
\end{equation}  

In anticipation of the scaled Korn inequality stated in Section \ref{subsection::KornInequalityInThinRods} we introduce a scaled inner product on $\Rntimesn$
\begin{equation*}
A\Colonh B := \frac{1}{h^2} \Sym A : \Sym B + \Skew A : \Skew B
\end{equation*}
for all $A$, $B\in\Rntimesn$ and $h>0$. The corresponding norm is denoted by $|A|_h := \sqrt{A\Colonh A}$. For $W\in \mathcal{L}^d(\Rntimesn)$ we define the induced scaled norm by 
\begin{equation*}
|W|_h := \sup_{|A_j|_h \leq 1, j =\{1,\ldots, d\}} |W(A_1, \ldots, A_d)|
\end{equation*}
As $|A|_h \geq |A|_1 = |A|$ for all $A\in\Rntimesn$ it follows that $|W|_h \leq |W|_1 =: |W|$ for all $W\in \mathcal{L}^d(\Rntimesn)$ and $0 < h \leq 1$. 

The scaled $L^p$-spaces are defined as follows
\begin{equation*}
\|W\|_{L^p_h(U, \mathcal{L}^d(\Rntimesn))} = \|W\|_{L^p_h(U)} = \left( \int_U |W(x)|_h^p dx\right)^{\frac{1}{p}}
\end{equation*}
if $p\in [1, \infty)$, where $U\subset\R^d$ is measurable. Thus $\|W\|_{L^p_h(U; \mathcal{L}^d(\Rntimesn))} \leq \|W\|_{L^p(U;\mathcal{L}^d(\Rntimesn))}$ for all $0<h\leq 1$. 
The scaled norm for $f\in L^p(U, \Rntimesn)$ is defined in the same way
\begin{equation*}
\|f\|_{L^p_h(U,\Rntimesn)} = \|f\|_{L^p_h(U)} = \left( \int_U |f(x)|_h^p dx\right)^{\frac{1}{p}}
\end{equation*}
and the inequality holds the other way round
\begin{equation*}
\|f\|_{L^p_h(U; \Rntimesn)} \geq \|f\|_{L^p(U;\Rntimesn)}.
\end{equation*}
Throughout this work we denote by $S\subset\R^2$ a smooth domain and $\Omega_h := (0,L) \times hS \subset \R^3$ for $h \in (0,1]$ and $L>0$ some length in $\R$. As an abbreviation we will write $\Omega := \Omega_1$. We assume that $S$ satisfies $|S|=1$,
\begin{equation}
	\int_S x_2 x_3 dx' = 0
	\label{globalEqualitySecondMomentZero}
\end{equation}
and 
\begin{equation}
	\int_S x_2 dx' = \int_S x_3 dx' = 0
	\label{globalEqualityFirstMomentsZero}
\end{equation}
where $x' := (x_2, x_3)\subset \R^2$. This can always be achieved via a scaling, translation and rotation. Furthermore, we denote by $\nabla_h$ the scaled gradient defined as
\begin{equation}
	\nabla_h = \bigg(\partial_{x_1}, \frac{1}{h}\partial_{x_2}, \frac{1}{h}\partial_{x_3}\bigg)^T.
\end{equation}
The respective scaled and unscaled gradients in only $x':=(x_2, x_3)$ direction are denoted as follows
\begin{equation*}
	\nabla_{h, x'} := \bigg(\frac{1}{h}\partial_{x_2}, \frac{1}{h}\partial_{x_3}\bigg)^T \quad\text{ and }\quad \nabla_{x'} := \big(\partial_{x_2}, \partial_{x_3}\big)^T.
\end{equation*}
The standard notation $H^k(\Omega)$ and $H^k(\Omega; X)$ is used for $L^2$-Sobolev spaces of order $k\in\N$ with values in $\R$ and some space $X$, respectively. Moreover, we denote for $m\in\N$
\begin{align*}
H^m_{per}(\Omega) := \Big\{f\in H^m(\Omega) \;:\; \partial^\alpha_x f|_{x_1 = 0} = \partial^\alpha_x f|_{x_1 = L},\;\text{for all } |\alpha| \leq m-1 \Big\}.
\end{align*}
A subscript $(0)$ on a function space will always indicate that elements have zero mean value, e.g., for $g\in H^1_{(0)}(U)$ we have 
\begin{equation}
\int_U g(x) dx = 0
\end{equation}
where $U\subset\R^n$ is open and bounded. In various estimates we will use an anisotropic variant of $H^k(\Omega)$, as we will have more regularity in lateral direction. Therefore we define
\begin{align*}
H^{m_1, m_2} &(\Omega) := \Big\{u\in L^2(\Omega) \;:\; \partial_{x_1}^l \nabla_{x}^k u\in L^2(\Omega) \text{ for all } k=0,\ldots, m_1, l=0,\ldots,m_2\\
&\partial_{x_1}^q \partial_{x}^\alpha u\Big|_{x_1 = 0} = \partial_{x_1}^q\partial_{x}^\alpha u\Big|_{x_1 = L} \text{ for all } q = 0,\ldots, m_1, |\alpha|\leq m_2 \text{ with } q + |\alpha| \leq m_1 + m_2 -1 \Big\}
\end{align*}
where $m_1,m_2\in\N_0$, the inner product is given by
\begin{equation*}
(f,g)_{H^{m_1, m_2}(\Omega)} = \sum_{k=0,\ldots,m_1; l=0,\ldots m_2} \Big(\partial_{x_1}^l \nabla_x^k f, \partial_{x_1}^l \nabla_x^k g \Big)_{L^2(\Omega)}.
\end{equation*}
Furthermore we will use the scaled norms
\begin{align*}
\|A\|_{H^m_h(\Omega)} &:= \left(\sum_{|\alpha| \leq m} \|\partial^\alpha_x A\|_{L^2_h(\Omega)}^2\right)^{\frac{1}{2}}\\
\|B\|_{H^{m_1, m_2}_h(\Omega)} &:= \left(\sum_{k=0,\ldots,m_1; l=0,\ldots,m_2} \|\partial_{x_1}^l \nabla_x^k B\|^2_{L^2_h(\Omega)}\right)^{\frac{1}{2}}.
\end{align*}
for $A\in H^m(\Omega;\R^{n\times n})$ and $B\in H^{m_1,m_2}(\Omega;\R^{n\times n})$ and $n\in\N$. As an abbreviation we denote for $u\in H^k(\Omega;\R^3)$ the symmetric scaled gradient by $\e_h(u):=\Sym(\nabla_h u)$ and $\e(u) = \e_1(u) = \Sym(\nabla u)$.

The space of periodic functions can be defined in an equivalent way, which is in some situations more convenient 
\begin{equation*}
\tilde{H}^m_{per}(\Omega) := \Big\{f\in H^m_{loc}(\R\times\bar{S}) \;:\; f(x_1, x')= f(x_1 + L, x') \text{ almost everywhere}\Big\}
\end{equation*}
equipped with the standard $H^m(\Omega)$-norm. As the maps $f\mapsto f|_{\Omega}\colon \tilde{H}^m_{per}(\Omega) \to H^m_{per}(\Omega)$ and $f\mapsto f_{per}\colon H^m_{per}(\Omega) \to \tilde{H}^m_{per}(\Omega)$ are isomorphisms, we identify $\tilde{H}^m_{per}(\Omega)$ with $H^m_{per}(\Omega)$. With this definition we obtain immediately that $C^\infty_{per}(\Omega)$ is dense in $H^m_{per}(\Omega)$, because, as $S$ is smooth, there exists an appropriate extension operator and thus we can use a convolution argument. The following lemma provides the possibility to take traces for $u\in H^{0,1}(\Omega)$, more precisely
\begin{lem}
	The operator $\tr{a}\colon H^{0,1}(\Omega) \to L^2(S)$, $u\mapsto u|_{x_1=a}$ is well defined and bounded.
\end{lem}
\begin{proof}
	This is an immediate consequence of the embedding
	\begin{equation*}
	H^{0,1}(\Omega) = H^1(0,L;L^2(S)) \hookrightarrow BUC([0,L];L^2(S))
	\end{equation*}
	where $BUC([0,L];X)$ is the space of all uniformly continuous functions $f\colon [0,L] \to X$ for some Banach space $X$. 
\end{proof}

\subsection{The Strain Energy Density $W$}
\label{subsection:StrainEnergyDensityW}
We investigate the mathematical assumptions and resulting properties of the strain-energy density $W$ in three dimensions. Assume $W\colon \Rtimes\to [0,\infty]$ satisfies the following conditions:
\begin{enumerate}
	\item[(i)]{$W\in C^\infty(B_\delta(Id); [0,\infty))$ for some $\delta >0$;}
	\item[(ii)] {$W$ is frame-invariant, i.e., $W(RF) = W(F)$ for all $F\in\Rtimes$ and $R\in SO(3)$;}
	\item[(iii)] {there exists $c_0 > 0$ such that $W(F) \geq c_0 \dist^2(F, SO(3))$ for all $F\in\Rtimes$ and $W(R) = 0$ for every $R\in SO(3)$.}
\end{enumerate}
\begin{bem} \label{RemarkdPropertiesOfElasticEnergyDensity}
	First of all we note that $W$ has a minimum at the identity, as $W(Id) = 0$ and $W(F)\geq 0$ for all $F\in\Rtimes$. Hence, we have $DW(Id)[G] = 0$ for all $G\in\Rtimes$. With the frame invariance symmetry of the Piola-Kirchhoff stress follows, i.e., for all $F\in\Rtimes$: $DW(F)F^T = F DW(F)^T$. Moreover, using the frame invariance we can deduce
%
%
	\begin{equation}
	D^2W(Id) [G, G] = D^2W(Id) [\sym(G), \sym(G)] \geq c_1 |\sym(G)|^2
	\label{globalInequalityCoercivityOfDW}
	\end{equation}
	for some $c_1>0$ and all $G\in\Rtimes$.
\end{bem}
\begin{bem}
	Using the identification \eqref{IdentificationL2withTensor4Order}, we can find $B^{\alpha \beta} = (b^{\alpha\beta}_{ij})_{i,j=1,2,3}\in \Rtimes$, for $\alpha, \beta\in \{1,2,3\}$ such that
	\begin{equation*}
	D^2W(Id) [X, Y] = D^2W(Id)X : Y = \sum_{\alpha,\beta, i, j =1}^3 b^{\alpha\beta}_{ij} x_{i\alpha}y_{j\beta}
	\end{equation*}
	for all $X, Y\in\Rtimes$. Therefore 
	\begin{equation*}
	D^2W(Id) [X] \nu = (D^2W(Id)X)\nu = \bigg(\sum_{\alpha,\beta, i, j =1}^3 b^{\alpha}_{ij} x_{i\alpha}\nu_{j}\bigg)_{\beta=1,2,3}
	\end{equation*}
	for $X\in\Rtimes$ and $\nu\in\R^3$. Hence, we obtain with \eqref{globalInequalityCoercivityOfDW} that $b^{\alpha\beta}_{ij} = b^{ji}_{\beta\alpha}$. In order to see this we choose $X = e_i\otimes e_\alpha - e_\alpha\otimes e_i$, $Y = e_j\otimes e_\beta$ and $X = e_i\otimes e_\alpha$, $Y=e_j\otimes e_\beta  - e_\beta\otimes e_j$, respectively. Thus either $\operatorname{sym}(X)= 0$ or $\operatorname{sym}(Y) = 0$ and with the symmetry property of Remark \ref{RemarkdPropertiesOfElasticEnergyDensity} it follows 
	\begin{equation}
	0 = D^2W(Id)[X, Y] = b^{\alpha\beta}_{ij} - b^{\beta i}_{\alpha j} = b^{\beta\alpha}_{ij} - b^{j\alpha}_{i\beta}.
	\label{EquationIndexSwichOfSystemOperator}
	\end{equation}
	For later use we introduce
	\begin{equation*}
	(D^2W(Id))^\approx := (B^{\alpha\beta})_{\alpha=2,3}^{\beta=2,3}.
	\end{equation*}
\end{bem}

Let $\tilde{W}\colon\Rtimes\to [0,\infty]$ be defined by $\tilde{W}(F) := W(Id+F)$. The results of Remark \ref{RemarkdPropertiesOfElasticEnergyDensity} therefore hold for $\tilde{W}$ as well, i.e., 
\begin{equation*}
D^2\tilde{W}(0) [G, G] = D^2\tilde{W}(0) [\e(G), \e(G)] \geq c_1 |\e(G)|^2
\end{equation*}
and 
\begin{equation*}
D^2\tilde{W}(0) [a\otimes b, a\otimes b] \geq c |a|^2 |b|^2 \quad\text{for all } a, b\in\R^3.
\end{equation*}

The following lemma provides an essential decomposition of $D^3\tilde{W}$.
\begin{lem}
	There is some constant $C>0$, $\e >0$ and $A\in C^\infty(\overline{B_\e(0)}; \mathcal{L}^3(\Rntimesn))$ such that for all $G\in\Rntimesn$ with $|G| \leq\e$ we have
	\begin{equation*}
	D^3\tilde{W}(G) = D^3\tilde{W}(0) + A(G)
	\end{equation*}
	where
	\begin{alignat}{2}
	|D^3\tilde{W}(0)|_h &\leq Ch &\qquad & \text{for all } 0<h\leq 1, \label{Lemma2.6_1}\\
	|A(G)| &\leq C|G| &&\text{for all } |G|\leq\e.\label{Lemma2.6_2}
	\end{alignat}
	\label{Lemma_DecompD3TildeW}
\end{lem}
\begin{proof}
	We refer to \cite[Lemma 2.6]{AbelsMoraMueller}.
\end{proof}
With this we can prove the following bound for $D^3\tilde{W}$.
\begin{Coro}
	There exist $C$, $\e > 0$ such that
	\begin{equation}
	\|D^3\tilde{W}(Z)(Y_1, Y_2, Y_3)\|_{L^1(\Omega)} \leq Ch \|Y_1\|_{H^2_h(\Omega)}\|Y_2\|_{L^2_h(\Omega)}\|Y_3\|_{L^2_h(\Omega)}
	\label{Abels_(2.15)}
	\end{equation}
	for all $Y_1\in H^2(\Omega, \Rntimesn)$, $Y_2$, $Y_3\in L^2(\Omega; \Rntimesn)$, $0<h\leq 1$ and $\|Z\|_{L^\infty(\Omega} \leq \min\{\e, h\}$ and
	\begin{equation}
	\|D^3\tilde{W}(Z)(Y_1, Y_2, Y_3)\|_{L^1(\Omega)} \leq Ch \|Y_1\|_{H^1_h(\Omega)} \|Y_2\|_{H^1_h(\Omega)} \|Y_3\|_{L^2_h(\Omega)}
	\label{Abels_(2.16)}
	\end{equation}
	for all $Y_1$, $Y_2\in H^1(\Omega, \Rntimesn)$, $Y_3\in L^2(\Omega; \Rntimesn)$, $0<h\leq 1$ and $\|Z\|_{L^\infty(\Omega} \leq \min\{\e, h\}$ and 
	\begin{equation}
	\|D^3\tilde{W}(Z)(Y_1, Y_2, Y_3)\|_{L^1(\Omega)} \leq Ch \bigg\|\bigg(Y_1, \frac{1}{h} \operatorname{sym}(Y_1)\bigg)\bigg\|_{L^\infty(\Omega)} \|Y_2\|_{H^1_h(\Omega)} \|Y_3\|_{L^2_h(\Omega)}
	\label{globalBoundednessOfD3W(Z)(Y1,Y2,Y3)LInftyVersion}
	\end{equation}
	for all $Y_1\in L^\infty(\Omega, \Rntimesn)$, $Y_2$, $Y_3\in L^2(\Omega; \Rntimesn)$, $0<h\leq 1$ and $\|Z\|_{L^\infty(\Omega} \leq \min\{\e, h\}$.
	\label{Corollary_L1BoundsD3TildeW}
\end{Coro}
\begin{proof}
	The inequalities follow directly from Lemma \ref{Lemma_DecompD3TildeW} and Hölder's inequality.
\end{proof}

\subsection{Korn's Inequality in Thin Rods}
\label{subsection::KornInequalityInThinRods}
In order to derive sharp estimates based on the linearised system, we need a good understanding on how the scaled gradient $\nabla_h g$ of a function $g\in H^1_{per}(\Omega)$ can be bounded by the scaled symmetric gradient $\e_h(g)$. As rigid motions $x\mapsto \alpha x^\perp$ for $\alpha\in\R$ arbitrary are admissible functions in $H^1_{per}(\Omega)$ we can not expect that the full scaled gradient is bounded by $\e_h(g)$. Moreover, a quantitative, sharp understanding of the dependency of a possible prefactor from the small parameter $h$ is essential.

\begin{lem}
  There exists a constant $C=C(\Omega)>0$ such that for all $0 < h\leq 1$ and $u\in H^1_{per}(\Omega;\R^3)$
	\begin{equation}
	\bigg\|\nabla_h u - \frac{1}{h}B(u)\bigg\|_{L^2(\Omega)} \leq C\bigg\|\frac{1}{h} \e_h(u)\bigg\|_{L^2(\Omega)},
	\label{Korn}
	\end{equation}
        where
	\begin{equation}
	B(u) =
	\begin{pmatrix}
	0&0&0\\
	0&0& a(u)\\
	0& - a(u)& 0
	\end{pmatrix}
	\label{KornStructureS}
	\end{equation}
	with $a(u) = \frac{1}{|\Omega|}\int_\Omega (\partial_{x_3} u_2(x) - \partial_{x_2} u_3(x)) dx$.
	\label{KornIneqS}
\end{lem}
\begin{proof}
	The proof is similar to \cite[Lemma 2.1]{AbelsMoraMueller} and is done in \cite[Lemma 2.4.4]{AmeismeierDiss}
\end{proof}

\begin{lem}[Korn inequality in integral form]\ \\
	For all $0 < h\leq 1$ and $u\in H^1_{per}(\Omega;\R^3)$, there exists a constant $C_K=C_K(\Omega)$, such that
	\begin{equation}
	\|\nabla_h u\|_{L^2(\Omega)}\leq \frac{C_K}{h}\bigg( \|\e_h(u)\|_{L^2(\Omega)} + \bigg| \int_\Omega u \cdot x^\perp dx\bigg|\bigg),
	\label{KornMeanValue}
	\end{equation}
	where $x^\perp = (0, -x_3, x_2)^T$.
	\label{LemmaKornIntegralForm}
\end{lem}
\begin{proof}
	First we note that we can reduce to the case of mean value free functions. If \eqref{KornMeanValue} holds for mean value free functions and $\int_\Omega u \neq 0$, we can consider $v := u - \tfrac1L \int_\Omega u$. Then it follows
	\begin{align*}
	\|\nabla_h u\|_{L^2(\Omega)} &= \|\nabla_h v\|_{L^2(\Omega)} \leq \frac{C_K}{h}\bigg( \|\e_h(v)\|_{L^2(\Omega)} + \bigg| \int_\Omega v \cdot x^\perp dx\bigg|\bigg)\\
	& \leq \frac{C_K}{h}\bigg( \|\e_h(u)\|_{L^2(\Omega)} + \bigg| \int_\Omega u \cdot x^\perp dx\bigg| + \bigg|\int_\Omega udx \cdot \int_\Omega x^\perp dx\bigg|\bigg)
	\end{align*}
	and thus \eqref{KornMeanValue} holds for $u$, as $\int_\Omega x^\perp dx = 0$.
	
	In the following we will argue by contradiction and therefore assume that \eqref{KornMeanValue} does not hold. Thus we can find a monotone sequence $h_k \to 0$ for $k\to\infty$ and $(u^{h_k})_{k\in\N}\subset H^1_{per, (0)}(\Omega, \R^3)$ such that
	\begin{equation}
	1= \|\nabla_{h_k} u^{h_k}\|_{L^2(\Omega)} \geq \frac{k}{{h_k}}\bigg( \|\e_{h_k} (u^{h_k})\|_{L^2(\Omega)} + \bigg|\int_\Omega u^{h_k} \cdot x^\perp dx\bigg|\bigg).
	\label{KornContradictionInequality}
	\end{equation}
	For sake of readability, we write $h$ instead of $h_k$ in the following calculations. From \eqref{KornContradictionInequality} it follows
	\begin{equation*}
	\frac{1}{h}\|\e_h(u^h)\|_{L^2(\Omega)} \leq \frac{1}{k} \to_{k\to\infty} 0,
	\end{equation*}
	which implies, using Lemma \ref{KornIneqS}
	\begin{equation*}
	\Big\|\nabla_h u^h - \frac{1}{h}B(u^h)\Big\|_{L^2(\Omega)} \leq C\Big\|\frac{1}{h}\e_h(u^h)\Big\|_{L^2(\Omega)} \to 0.
	\end{equation*}
	Thus $\tfrac{1}{h}B(u^h)$ is bounded in $L^2(\Omega)$ and therefore bounded in $\Rntimesn_{skew}$. Using a subsequence, also denoted by $u^h$, it follows $\tfrac{1}{h}B(u^h) \to \bar{B}$ for $h\to 0$. As a consequence of \eqref{KornStructureS} the structure of $\bar{B}$ is given by
	\begin{equation*}
	B = 
	\begin{pmatrix}
	0&0&0\\
	0&0&-\bar{a}\\
	0&\bar{a}&0
	\end{pmatrix}
	\end{equation*}
	where $\bar{a}\neq 0$ as, $\nabla_h u^h \to \bar{B}$ in $L^2(\Omega)$ and $\|\nabla_h u^h\|_{L^2(\Omega)} = 1$
	
	Define now 
	\begin{equation*}
	w^h_l(x') := \frac{1}{L} \int_0^L \frac{u^h_l(x)}{h} dx_1
	\end{equation*}
	with $x'\in S$ and $l=2,3$. Then
	\begin{equation*}
	\nabla_{x'} w^h \to \begin{pmatrix}
	0& -\bar{a}\\
	\bar{a}&0
	\end{pmatrix}
	\end{equation*}
	in $L^2(S)$ and $\int_S w^h dx' = 0$. Thus using the Poincaré inequality it follows 
	\begin{equation*}
	\|w^h\|_{H^1(S)} \leq C\|\nabla_{x'} w^h\|_{L^2(S)} \leq C\|\nabla_h u^h\|_{L^2(\Omega)} \leq C
	\end{equation*}
	Thus, there exists a subsequence $w^h \wcon w$ in $H^1(S)$ and $w^h \to w$ in $L^2(S)$. Choose $S'\subset \bar{S'} \subset S$ and $\delta>0$ such that $\delta\leq \operatorname{dist}(S', \partial S)$. Then
	\begin{equation*}
	\frac{w^h_2(x_2 + \delta, x_3) - w^h_2(x_2, x_3)}{\delta} = \frac{1}{\delta}\int_0^\delta \partial_2 w^h_2(x_2+\tau, x_3)d\tau.
	\end{equation*}
	From the above we know 
	\begin{equation*}
	\frac{w^h_2(x_2 + \delta, x_3) - w^h_2(x_2, x_3)}{\delta}\to \frac{w_2(x_2 + \delta, x_3)- w_2(x_2, x_3)}{\delta} \quad \text{ for } h\to 0
	\end{equation*} 
	in $L^2(S')$ and thus a subsequence converges pointwise almost everywhere. For the right hand side we have that $\partial_{x_2} w^h_2 \to 0$ for $h\to 0$ in $L^2(S)$ because of
	\begin{align*}
		\bigg\|\frac{1}{\delta} \int_0^\delta \partial_2 w^h_2(x_2 + \tau, x_3)d\tau\bigg\|_{L^2(S')}^2 & = \int_{S'}\bigg|\frac{1}{\delta}\int_0^\delta\partial_2 w^h_2(x_2 + \tau, x_3) d\tau \bigg|^2 dx'\\
		& \leq \int_{S'}\frac{1}{\delta} \int_0^\delta |\partial_2 w^h_2(x_2 + \tau, x_3) |^2d\tau dx\\
		& =\frac{1}{\delta} \int_0^\delta \|\partial_2 w^h_2(\cdot + \tau e_2)\|_{L^2(S')} d\tau \to 0
	\end{align*}
	where we used Hölder's inequality and the dominated convergence theorem for 
	\begin{equation*}
	\|\partial_2 w^h_2(\cdot + \tau e_2)\|_{L^2(S')} \leq \|\partial_2 w^h_2\|_{L^2(S)} \to 0 \quad \text{ for } h\to 0.
	\end{equation*}
	Thus the mean value satisfies $\tfrac{1}{\delta}\int_0^\delta \partial_2 w^h_2(x_2 + \tau, x_3)d\tau \to 0$ in $L^2(S')$. Hence, as $S$ is a domain, $w_2$ is independent of $x_2$. Similarly one can show that $w_3$ is independent of $x_3$. Furthermore
	\begin{equation*}
	\frac{w^h_2(x_2, x_3 + \delta) - w^h_2(x_2, x_3)}{\delta} = \frac{1}{\delta}\int_0^\delta \partial_3 w_2^h(x_2, x_3 + \tau)d\tau
	\end{equation*}
	where the right-hand side converges to the constant $-\bar{a}$ in $L^2(S')$. Since $x'\in S'$ was chosen arbitrarily it follows
	\begin{equation*}
	w_2(x') = w^0_2 - \bar{a}x_3,
	\end{equation*}
	where $w^0_2\in\R$ is a constant. Applying the same argument to $w_3$ it follows that 
	\begin{equation*}
	w(x) = \left(\begin{array}{c}
	w^0_2\\ w^0_3
	\end{array} \right)
	+ \bar{a} \left(\begin{array}{c}
	-x_3\\
	x_2
	\end{array} \right).
	\end{equation*}
	Hence
	\begin{align*}
	\bigg| \int_{\Omega} \frac{u^h}{h} \cdot x^\perp dx\bigg| &= L\bigg| \frac{1}{L} \int_0^L \int_{S} \frac{u^h}{h}\cdot x^\perp dx\bigg|\\
	& = L\bigg| \int_S w^h \cdot x^\perp dx\Big| \to L\bar{a} \Big| \int_S x_2^2 + x_3^2dx\bigg| \neq 0
	\end{align*}
	as $\bar{a}\neq 0$. But this contradicts \eqref{KornContradictionInequality}.
\end{proof}

Later we will need the Korn inequality in two dimensions without scaling while analysing a stationary problem associated with the linearised equation.
\begin{Coro}[Korn inequality in two dimensions]~\\
	 There exists a constant $C=C(\Omega) > 0$ such that for all $u\in H^1(S; \R^2)$ 
	\begin{equation}
	\|\nabla u\|_{L^2(S)} \leq C\Big(\|\e(u)\|_{L^2(S)} + \bigg|\int_S u\cdot x^\perp dx\bigg|\Big)
	\label{InequalityKornTwoDimensions}
	\end{equation}
	where in this situation $x^\perp := (-x_3, x_2)^T$.
	\label{LemmaKornInequalityInTwoDimensions}
\end{Coro}
\begin{proof}
	We can deduce \eqref{InequalityKornTwoDimensions} from Lemma \ref{LemmaKornIntegralForm}. Let $u\in H^1(S;\R^2)$ and define $\tilde{u}\in H^1_{per}(\Omega;\R^3)$ by
	\begin{equation*}
	\tilde{u}(x) := \frac{1}{\sqrt{L}}
	\begin{pmatrix}
	0\\u(x')
	\end{pmatrix}.
	\end{equation*}
	Then it follows from \eqref{KornMeanValue} for $h=1$ applied to $\tilde{u}$
	\begin{align*}
	\|\nabla u\|_{L^2(S)} = \|\nabla\tilde{u}\|_{L^2(\Omega)} &\leq C\bigg( \|\e(\tilde{u})\|_{L^2(\Omega)} + \bigg| \int_\Omega \tilde{u} \cdot x^\perp dx\bigg|\bigg) \\
	&= C\bigg( \|\e(u)\|_{L^2(S)} + \bigg| \int_S u \cdot x^\perp dx\bigg|\bigg).
	\end{align*}

\end{proof}

\section{Large Time Existence for the Non-linear System}
\label{section::LargeTimeExistence}
\subsection{Main Result}
\label{subsection::MainResult}
We consider the following system:
\begin{align}
\partial_t^2 u_h - \frac{1}{h^2}\operatorname{div}_h \Big(D\tilde{W}(\nabla_h u_h)\Big) &= h^{1+\theta} f_h \quad \text{in } \Omega\times [0, T) \label{NLS1}\\
D\tilde{W}(\nabla_h u_h)\nu|_{(0, L)\times \partial S} &= 0\label{NLS2}\\
u_h \text{ is $L$-periodic} &\text{ w.r.t. } x_1 \label{NLS3}\\
(u_h, \partial_t u_h)|_{t=0} &= (u_{0,h}, u_{1,h}) \label{NLS4}
\end{align}
where for convenience $\theta = \alpha - 3\geq 1$. As an abbreviation we denote in the following $z= (t,x_1)$. Note that $D^2 W(Id)$ satisfies the Legendre-Hadamard condition and Lemma~\ref{Lemma_DecompD3TildeW} holds. The main result is:
\begin{theorem}
	Let $\theta \geq 1$, $0< T< \infty$, $M>0$, $f_h\in W^3_1(0,T; L^2(\Omega)) \cap W^1_1(0,T; H^2_{per}(\Omega))$, $h\in (0,1]$ and $u_{0,h}\in H^4_{per}(\Omega)$, $u_{1,h}\in H^3_{per}(\Omega)$ such that
	\begin{align*}
	D\tilde{W}(\nabla_h u_{0,h})\nu|_{(0,L)\times\partial S} = D^2\tilde{W}(\nabla_h u_{0,h})[\nabla_h u_{1,h}]\nu|_{(0,L)\times\partial S} = 0,\\
	(D^2\tilde{W}(\nabla_h u_{0,h})[\nabla_h u_{2,h}]+ D^3\tilde{W}(\nabla_h u_{0,h})[\nabla_h u_{1,h}, \nabla_h u_{1,h}])\nu|_{(0,L)\times\partial S} = 0,
	\end{align*}
	where 
	\begin{align*}
	u_{2,h} &= h^{1+\theta} f_h|_{t=0} + \frac{1}{h^2}\divh(D\tilde{W}(\nabla_h u_{0,h}))\\
	u_{3,h} & = h^{1+\theta} \partial_t f_h|_{t=0} + \frac{1}{h^2} \divh(D^2\tilde{W}(\nabla_h u_{0,h})\nabla_h u_{1,h})\\
	u_{4,h} & = h^{1+\theta} \partial_t^2 f_h|_{t=0} + \frac{1}{h^2} \divh(D^2\tilde{W}(\nabla_h u_{0,h})\nabla_h u_{2,h})\\
	& \qquad + \frac{1}{h^2} \divh(D^3\tilde{W}(\nabla_h u_{0,h})[\nabla_h u_{1,h}, \nabla_h u_{1,h}]).
	\end{align*}
	Moreover we assume for the initial data
	\begin{gather}
	\Big\|\frac{1}{h}\e_h(u_{0,h})\Big\|_{H^2} + \max_{k=0,1,2} \Big\|\Big(\frac{1}{h}\e_h(u_{1+k, h}), \partial_{x_1}\frac{1}{h}\e_h(u_{k,h}), u_{2+k,h}\Big)\Big\|_{H^{2-k}} \leq Mh^{1+\theta} \label{AssumptionOnInitialDataU1}\\
	\Big\|\nabla_h^2 u_{0,h}\Big\|_{H^1} + \max_{k=0,1} \Big\|\Big(\nabla_h^2 u_{1+k, h}, \partial_{x_1}\nabla_h^2 u_{k,h}\Big)\Big\|_{H^{1-k}} \leq Mh^{1+\theta} \label{AssumptionOnInitialDataU2}\\
	\max_{k=0,1,2,3} \bigg|\frac{1}{h} \int_\Omega u_{k,h}\cdot x^\perp dx\bigg| \leq Mh^{1+\theta}.\label{AssumptionOnInitialDataU3}
	\end{gather}
	and for the right hand side
	\begin{gather}
	\max_{|\alpha|\leq 1} \bigg(\|\partial_z^\alpha f_h\|_{W^2_1(L^2)} + \|\partial_z^\alpha f_h\|_{W^1_\infty(L^2)\cap W^1_1(H^{0,1})} + \|\partial_z^\alpha f_h\|_{L^\infty(H^1)}\bigg) \leq M \label{AssumptionsOnNonlinearf}\\
	\max_{\sigma =0,1,2} \bigg\| \frac{1}{h} \int_\Omega \partial_t^\sigma f_h\cdot x^\perp dx\bigg\|_{C^0([0,T])} \leq M \label{AssumptionsOnRotationOfNonlinearf}
	\end{gather}
	uniformly in $0<h\leq 1$.
	Then there exists $h_0\in (0,1]$ and $C>0$ depending only on $M$ and $T$ such that for every $h\in (0,h_0]$ there is a unique solution $u_h\in \bigcap_{k=0}^4 C^k([0,T]; H^{4-k}_{per})$ of \eqref{NLS1}--\eqref{NLS4} satisfying 
	\begin{align}
	&\max_{\substack{|\alpha| \leq 1, |\beta|\leq 2, |\gamma|\leq 1\\\sigma=0,1,2}} \bigg(\Big\|\Big(\partial_t^2\partial_t^\sigma u_h, \nabla_{x,t}^\beta\frac{1}{h}\e_h(\partial_z^\alpha u_h), \nabla_{x,t}^\gamma\nabla_h^2 \partial_z^\alpha u_h\Big)\Big\|_{C^0([0,T], L^2)}\notag \\
	&\hspace{3cm}+ \bigg\|\frac{1}{h} \int_\Omega \partial_z^{\alpha + \beta} u_h \cdot x^\perp dx\bigg\|_{C^0([0,T])}\bigg) \leq C h^{1+\theta} \label{GlobalMainTheoremInequality}
	\end{align}
	uniformly in $0<h\leq h_0$.
	\label{MainTheorem}
\end{theorem}
For fixed $h>0$ short time existence is already known via the methods of \cite{Koch} if the fixed time of existence is replaced by some $h$ dependent maximal time $T(h)>0$. Hence only the uniform estimates for $u_h$ and that $T$ does not depend on $h$ has to be shown. In detail, one obtains from \cite{Koch}:
\begin{theorem}
	Let the assumption of Theorem \ref{MainTheorem} hold true. Then for any $0 < h\leq 1$ there exists a neighbourhood $U_h\subset\Rtimes$ of $0$ and some $T_{max}(h) > 0$ such that \eqref{NLS1}--\eqref{NLS4} has a unique solution $u_h\in \bigcap_{k=0}^4 C^k([0,T_{max}(h)); H^{4-k}_{per})$. If $T_{max}(h) <\infty$, then either $\{\nabla_h u_h(x,t) \;:\; x\in\overline{\Omega}, t\in[0,T_{max}(h))\}$ is not precompact in $U_h$ or
	\begin{equation*}
	\lim_{t\to T_{max}(h)} \int_0^t \|\nabla_{x,t}^2 u_h(s)\|_{L^\infty(\Omega)} ds = \infty.
	\end{equation*}
	\label{TheoremNonlinearShortTimeExistence}
\end{theorem}
\begin{bem}
	We will give a more precise explanation on how the results of \cite{Koch} are applied to our situtation in the appendix. At this point however we want to mention that the neighbourhood $U_h$ can be chosen as 
	\begin{equation*}
		U_h := \bigg\{A\in\Rtimes \;:\; \bigg|\bigg(A, \frac{1}{h}\operatorname{sym}(A)\bigg)\bigg| \leq \e h\bigg\}
	\end{equation*}
	where $\e > 0$ is sufficiently small. With this it follows that as long as $\nabla_h u_h(x,t) \in U_h$ is satisfied the necessary weak coercivity holds, cf. Section \ref{subsec::UniformEstimatesForLinearisedSystem}.
\end{bem}

The strategy for proving the main result is as follows. In a first step we will derive precise estimates for solutions of the linearisation of \eqref{NLS1}--\eqref{NLS4} under the assumption that $u_h$ is small in appropriate norms. To this end we use the natural boundary conditions, differentiate tangentially and utilize the central estimate
\begin{equation}
	\frac{1}{h^2}\Big(D^2\tilde{W}(\nabla_h u_h)\nabla_h w, \nabla_h w\Big)_{L^2(\Omega)} \geq \frac{c_0}{2} \Big\|\frac{1}{h} \e_h(w)\Big\|_{L^2(\Omega)}^2 - CR\bigg|\frac{1}{h}\int_\Omega w\cdot x^\perp dx\bigg|^2, 
	\label{globalCentralCoerciveEstimateForAppendix}
\end{equation}
proven below. By differentiating \eqref{NLS1} in time and $x_1$ we obtain that the respective derivative of $u_h$ solves now the linearised system. Applying then the results of the first step we can deduce with the balance of angular momentum that the solutions are uniformly bounded in $h$ if the initial values and external force are sufficiently small.

\subsection{Uniform Estimates for Linearised System}
\label{subsec::UniformEstimatesForLinearisedSystem}
The linearised system for \eqref{NLS1}--\eqref{NLS4} is given by 
\begin{align}
\partial_t^2 w - \frac{1}{h^2}\operatorname{div}_h (D^2\tilde{W}(\nabla_h u_h) 
\nabla_h w) &= f\quad\text{in }\Omega\times [0, T) \label{linearEQ_1}\\
D^2\tilde{W}(\nabla_h u)[\nabla_h w]\nu &=0 \quad\text{on } (0,L)\times\partial S\times [0, T)\label{linearEQ_2}\\
w \text{ is $L$-periodic}& \text{ in $x_1$ coordinate}\label{linearEQ_3}\\
(w, \partial_t w)|_{t=0} &= (w_0, w_1).\label{linearEQ_4}
\end{align}

We want to show $h$-independent estimates for solutions of the linearised system. For this we assume that $u_h$ satisfies for $0 < h \leq 1$
\begin{align}
\sup_{\substack{|\alpha| \leq 1, |\beta|\leq 2\\\sigma=0,1,2}} \bigg(\Big\|\Big(\nabla_{x,t}^k\frac{1}{h} &\e_h(\partial_z^\alpha u_h), \nabla_{x,t}^k \nabla_h \partial_z^\alpha u_h\Big)\Big\|_{C^0([0,T];L^2(\Omega))} \notag\\
& + \bigg\|\frac{1}{h}\int_\Omega \partial_z^{\alpha+\beta} u_h\cdot x^\perp dx\bigg\|_{C^0([0,T])} \bigg) \leq Rh,
\label{AssumptionsOnU1}
\end{align}
where $R\in (0,R_0]$, with $R_0$ chosen later appropriately small. With this it follows that $u_h$ satisfies
\begin{equation}
\|\nabla_h u_h\|_{C^0([0,T]; H^2_h(\Omega))} + \Big\|\Big(\frac{1}{h} \e_h(u_h), \nabla_h u_h\Big)\Big\|_{C^0([0,T]; L^\infty(\Omega))} \leq CRh
\label{AssumptionsOnU2}
\end{equation}
and
\begin{equation}
\sup_{|\alpha| \leq 2} \bigg(\Big\|\Big(\frac{1}{h}\e_h(\partial_z^\alpha u_h), \nabla_h\partial_z^\alpha u_h\Big)\Big\|_{C^0([0,T]; H^1(\Omega))}  + \bigg\|\frac{1}{h} \int_\Omega\partial_z^\alpha u_h\cdot x^\perp dx\bigg\|_{C^0([0,T])}\bigg) \leq CRh.
\label{AssumptionsOnU3}
\end{equation}
Here $C>0$ is independent of $h$, $R$ and $R_0$. In the following we assume that $R_0$ is chosen so small such that $\tilde{W} \in C^\infty(\overline{B_{CR_0}(0)})$ and Lemma \ref{Lemma_DecompD3TildeW} is applicable.

Using Corollary \ref{Corollary_L1BoundsD3TildeW}, we obtain
\begin{align*}
\bigg| \frac{1}{h^2}\int_0^1 \Big(D^3\tilde{W}(\tau\nabla_h u_h)[\nabla_h u_h, \nabla_h v], \nabla_h w\Big)_{L^2(\Omega)}d\tau\bigg| &\leq \frac{C}{h}\|\nabla_h u_h\|_{H^2_h(\Omega)}\|\nabla_h v\|_{L^2_h(\Omega)}\|\nabla_h w\|_{L^2_h(\Omega)}\\
& \leq CR \|\nabla_h v\|_{L^2_h(\Omega)}\|\nabla_h w\|_{L^2_h(\Omega)}.
\end{align*}
uniformly in $v, w\in H^1(\Omega)$, $0\leq t\leq T$ and $h\in (0,1]$. Thus it follows 
\begin{align}
\frac{1}{h^2}&\Big(D^2\tilde{W}(\nabla_h u_h)\nabla_h v, \nabla_h v \Big)_{L^2(\Omega)}\notag\\
&= \frac{1}{h^2}(D^2\tilde{W}(0)\nabla_h v, \nabla_h v)_{L^2(\Omega)} + \frac{1}{h^2} \int_0^1 (D^3\tilde{W}(\tau\nabla_h u_h)[\nabla_h u_h, \nabla_h v], \nabla_h v)_{L^2(\Omega)}d\tau\notag\\
&\geq \frac{c_0}{h^2} \|\e_h(v)\|^2_{L^2(\Omega)} - \left|\frac{1}{h^2}\int_0^1 (D^3\tilde{W}(\tau\nabla_h u_h)[\nabla_h u_h, \nabla_h v], \nabla_h v)_{L^2(\Omega)}d\tau\right|\notag\\
& \geq c_0 \Big\|\frac{1}{h}\e_h(v)\Big\|^2_{L^2(\Omega)} - CR \|\nabla_h v\|^2_{L^2_h(\Omega)}.\label{Coercivity}
\end{align}
The structure of this subsection is that we will start with a general lemma providing a bound for derivatives in $z$ of $D^2\tilde{W}(\nabla_h u_h)$ in the case that $u_h$ satisfies \eqref{AssumptionsOnU1}. To obtain bounds on higher derivatives we investigate the static problem and apply these subsequently to the evolution equation. This approach leads to uniform estimates for the solution of the linearised system.
\begin{lem}
	Let \eqref{AssumptionsOnU1} hold true and let $t\in [0, T]$ and $0< R\leq R_0$. Then there is some $C=C(R_0)$ independent of $t,R,h$ such that
	\begin{equation}
	\left| \frac{1}{h^2} \left(\partial_{z}^\beta D^2\tilde{W}(\nabla_h u_h 
	(t)) \nabla_h w, \nabla_h v)\right)_{L^2(\Omega)}\right| \leq CR \|\nabla_h w\|_{H^{|\beta|-1}_h(\Omega)} \|\nabla_h v\|_{L^2_h(\Omega)}
	\label{equationAbschätzungenAbleitungD2W}
	\end{equation}
	for $1 \leq |\beta| \leq 3$ and $w\in H^{|\beta|}(\Omega)^3, v\in H^1(\Omega)^3$.
	\label{AbschätzungenAbleitungD2W}
\end{lem}
\begin{proof}
	If $|\beta| = 1$, we obtain by \eqref{AssumptionsOnU1} and \eqref{Abels_(2.15)}
	\begin{align*}
	\bigg| \frac{1}{h^2} \Big(\big(\partial_z^\beta D^2\tilde{W}(\nabla_h u_h (t))\big) \nabla_h w, \nabla_h v\Big)_{L^2(\Omega)}\bigg| & = \bigg|\frac{1}{h^2} \Big(D^3\tilde{W}(\nabla_h u_h (t))[\partial_z^\beta \nabla_h u, \nabla_h w] , \nabla_h v \Big)_{L^2(\Omega)}\bigg|\\
	& \leq \frac{C}{h} \|\nabla_h \partial_z^\beta u_h\|_{H^2_h(\Omega)} 
	\|\nabla_h w\|_{L^2_h(\Omega)}\|\nabla_h v\|_{L^2_h(\Omega)}\\
	& \leq CR \|\nabla_h w\|_{L^2_h(\Omega)} \|\nabla_h v\|_{L^2_h(\Omega)}
	\end{align*}
	If $|\beta| = 2$, it follows for $j$, $k\in \{0, 1\}$ chosen correctly
	\begin{equation}
	\partial_z^\beta D^2\tilde{W}(\nabla_h u_h) = D^3\tilde{W}[\partial_z^\beta\nabla_h u_h] + D^4\tilde{W}(\nabla_h u_h)[\partial_{z_j}\nabla_h u_h, \partial_{z_k}\nabla_h u_h].
	\end{equation}
	For the first term we use \eqref{Abels_(2.16)}
	\begin{equation*}
	\bigg|\frac{1}{h^2} \Big( D^3\tilde{W}(\nabla_h u_h)[\partial_z^\beta\nabla_h u_h, \nabla_h w], \nabla_h v\Big)_{L^2(\Omega)}\bigg| \leq CR \|\nabla_h w\|_{H^1_h(\Omega)} \|\nabla_h v\|_{L^2_h(\Omega)}
	\end{equation*}
	and as $D^4 \tilde{W}(\nabla_h u_h) \in C^0(\overline{\Omega}, \mathcal{L}^4(\Rtimes))$
	\begin{align*}
	\frac{1}{h^2}\int_\Omega \big| D^4\tilde{W}(\nabla_h u_h)[\partial_{z_j}&\nabla_h u_h, \partial_{z_k}\nabla_h u_h, \nabla_h w, \nabla_h v]\big|dx\\
	&\leq \frac{C}{h^2}\int_\Omega |\partial_{z_j}\nabla_h u_h| |\partial_{z_k}\nabla_h u_h| |\nabla_h w| |\nabla_h v| dx\\
	&\leq CR \int_\Omega |\nabla_h w||\nabla_h v| dx \leq CR\|\nabla_h w\|_{L^2(\Omega)}\|\nabla_h v\|_{L^2(\Omega)}
	\end{align*}
	Finally for $|\beta| = 3$ and $j$, $k$ and $l\in \{0,1\}$ such that $\partial_z^\beta = \partial_{z_j}\partial_{z_k}\partial_{z_l}$ we have
	\begin{align*}
	\partial_z^\beta D^2\tilde{W}(\nabla_h u_h) =
	D^3\tilde{W}(\nabla_h u_h)&[\partial_z^\beta\nabla_h u_h] + D^4\tilde{W}(\nabla_h u_h)[\partial_{z_l}\partial_{z_j}\nabla_h u_h, \partial_{z_k}\nabla_h u_h]\\
	&+ D^4\tilde{W}(\nabla_h u_h)[\partial_{z_j}\nabla_h u_h, \partial_{z_l}\partial_{z_k}\nabla_h u_h]\\
	& + D^4\tilde{W}(\nabla_h u_h)[\partial_{z_k}\partial_{z_j}\nabla_h u_h, \partial_{z_l}\nabla_h u_h]\\
	&+ D^5\tilde{W}(\nabla_h u_h)[\partial_{z_j}\nabla_h u_h, \partial_{z_k}\nabla_h u_h, \partial_{z_l} \nabla_h u_h]
	\end{align*}
	The fifth order term can be estimated in the same way as in the case of $|\beta|=2$. As $\partial_{z_l}\partial_{z_j}\nabla_h u_h \in H^1(\Omega) \hookrightarrow L^4(\Omega)$ and $\partial_{z_l}\nabla_h u\in H^2(\Omega) \hookrightarrow L^\infty(\Omega)$, it follows with the Hölder inequality that
	\begin{align*}
	\frac{1}{h^2}\int_\Omega \big|D^4\tilde{W}&(\nabla_h 
	u_h)[\partial_{z_l}\partial_{z_j}\nabla_h u_h, \partial_{z_l}\nabla_h u_h, 
	\nabla_h w, \nabla_h v]\big| dx\\
	&\leq \frac{C}{h^2}\int_\Omega |\partial_{z_l}\partial_{z_j}\nabla_h u_h||\partial_{z_l}\nabla_h u_h||\nabla_h w| |\nabla_h v|dx\\
	&\leq \frac{C}{h^2} \|\partial_{z_l}\partial_{z_j}\nabla_h u_h\|_{L^4(\Omega)}\|\partial_{z_l}\nabla_h u_h\|_{L^\infty(\Omega)}\|\nabla_h w\|_{L^4(\Omega)} \|\nabla_h v\|_{L^2(\Omega)}\\
	&\leq CR \|\nabla_h 
	w\|_{H^1(\Omega)}\|\nabla_h v\|_{L^2(\Omega)}.
	\end{align*}
	For the last term we use \eqref{Abels_(2.15)} 
	\begin{equation*}
	\bigg| \frac{1}{h^2}\Big(D^3\tilde{W}(\nabla_h u_h)[\partial_z^\beta \nabla_h u_h]\nabla_h w, \nabla_h v\Big)_{L^2(\Omega)}\bigg|
	\leq CR \|\nabla_h w\|_{H^2_h(\Omega)} \|\nabla_h v\|_{L^2_h(\Omega)}.
	\end{equation*}
\end{proof}
The first step to obtain higher regularity is done in the following theorem.
\begin{theorem}
	Assume $u_h$ satisfies \eqref{AssumptionsOnU1} and $k=0, 1$. Then there exist $C>0$ 
	and $R_0\in (0,1]$ such that, if $\varphi\in H^{2+k}_{per}(\Omega)$ solves for some $g\in H^k_{per}(\Omega)$ and $g_N\in L^2(0,L;H^{k+\frac{1}{2}}(\partial S))\cap H^{k}(0,L; H^{\frac{1}{2}}(\partial S))$
	\begin{equation}
	\left\{
	\begin{aligned}
		- \frac{1}{h^2} \operatorname{div}_h (D^2\tilde{W}(\nabla_h u_h)\nabla_h\varphi) &= g &&\quad\text{in } \Omega, \\
		D^2\tilde{W}(\nabla_h u_h)[\nabla_h \varphi]\nu\Big|_{(0,L)\times \partial S} &= g_N &&\quad\text{on } \partial\Omega,
	\end{aligned}
	\right.
	\label{LocalStationaryNonZeroNeumannBoundaryEquationHigherRegularity}
	\end{equation}
	then
	\begin{align}\nonumber
			\bigg\| \bigg(\nabla\frac{1}{h}\e_h(\varphi), \nabla_h^2
			\varphi\bigg)\bigg\|_{H^k(\Omega)} &\leq C\bigg(h^2\|g\|_{L^2(\Omega)} + \bigg\|\frac{1}{h} g_N\bigg\|_{L^2(0,L;H^{k+\frac{1}{2}}(\partial S))\cap H^{k}(0,L; H^{\frac{1}{2}}(\partial S))}\\
			& \quad + \bigg\|\frac{1}{h}\e_h(\varphi)\bigg\|_{H^{0,k+1}(\Omega)} + R
			\bigg|\frac{1}{h} \int_\Omega \varphi\cdot x^\perp dx\bigg|\bigg).
	\label{InequalityHigherRegularity}
	\end{align}
	\label{TheoremHigherRegularity}
\end{theorem}
\begin{proof}
	We start proving the result in the case $k=0$. Using the fundamental theorem of calculus it follows
	\begin{equation*}
	\operatorname{div}_h(D^2\tilde{W}(\nabla_h u_h)\nabla_h \varphi) = 
	\operatorname{div}_h (D^2\tilde{W}(0)\nabla_h \varphi) + \int_0^1 
	\operatorname{div}_h(D^3\tilde{W}(\tau \nabla_h u_h)[\nabla_h u_h, \nabla_h 
	\varphi])d\tau
	\end{equation*}
	and 
	\begin{align*}
	\divh(D^2\tilde{W}(0)\nabla_h \varphi) = \partial_{x_1}(D^2\tilde{W}(0)\nabla_h 
	\varphi)_1 &+ 
	\frac{1}{h}\operatorname{div}_{x'}(D^2\tilde{W}(0)^\sim \partial_{x_1}\varphi\otimes
	e_1) \\
	& + 
	\frac{1}{h^2}\operatorname{div}_{x'}(D^2\tilde{W}(0)^\approx \nabla_{x'}\varphi),
	\end{align*}
	where $(A)_k$ denotes the $k$th column of $A\in\R^{3\times 3}$. Moreover
	\begin{align}
	g_N &= D^2\tilde{W}(\nabla_h u_h)\nabla_h \varphi\nu\Big|_{(0,L)\times\partial S} = 
	\operatorname{tr}_{\partial S}(D^2\tilde{W}(\nabla_h u_h)\nabla_h \varphi)\nu\notag \\
	& = \operatorname{tr}_{\partial S}\bigg( D^2\tilde{W}(0)\nabla_h \varphi + 
	\int_0^1 D^3\tilde{W}(\tau\nabla_h u_h)[\nabla_h u, \nabla_h \varphi] 
	d\tau\bigg)\nu\notag\\
	& = \operatorname{tr}_{\partial 
		S}\bigg(\frac{1}{h}D^2\tilde{W}(0)^\approx\nabla_{x'}\varphi\bigg)\nu_{\partial S}\notag \\
	& \hspace{0.5cm} + \operatorname{tr}_{\partial 
			S}\underbrace{\bigg(D^2\tilde{W}(0)(\partial_{x_1}\varphi\otimes e_1) 
		+ \int_0^1 D^3\tilde{W}(\tau \nabla_h u_h)[\nabla_h u_h,\nabla_h \varphi] d\tau \bigg)}_{\displaystyle =: r_N}\bigg(\begin{matrix}0\\ \nu_{\partial S}\end{matrix}\bigg) \label{LocalDefinitionOfA_N}
	\end{align}
	where we have used that $\nu = (0, \nu_{\partial S}(x_2, x_3))^T$ and $\nu_{\partial S}$ is the outer unit normal on $\partial S$. Hence, we know that 
	\begin{equation}
		\varphi_{(0)}(x_1, \cdot) := \varphi(x_1, \cdot) -\frac{1}{\mu(S)} (\varphi, x^\perp)_{L^2(S)} x^\perp - (\varphi, 1)_{L^2(S)}
		\label{localDefitionSolutionReducedSystem}
	\end{equation}
	solves for almost all $x_1\in (0,L)$ the system
	\begin{equation}
	\left\{
	\begin{aligned}
	-\frac{1}{h^2} \operatorname{div}_{x'}\Big(D^2\tilde{W}(0)^\approx \nabla_{x'} \varphi(x_1, \cdot)\Big) &= \tilde{g}(x_1, \cdot) &&\quad\text{ in } S,\\
	\frac{1}{h}\Big(D^2\tilde{W}(0)^\approx \nabla_{x'} \varphi(x_1, \cdot)\Big) \nu_{\partial S}\bigg|_{\partial S} &= g_N(x_1, \cdot) - a_N(x_1, \cdot)&&\quad\text{ on } \partial S
	\end{aligned}
	\right.
	\label{reducedSystem}
	\end{equation}
	with $a_N := \operatorname{tr}_{\partial S} (r_N) (0, \nu_{\partial S})^T$,
	\begin{equation}
	\tilde{g} := h^2 g + \mathcal{R}(\varphi) + \int_0^1 
	\operatorname{div}_h(D^3\tilde{W}(\tau \nabla_h u_h)[\nabla_h u_h, \nabla_h 
	\varphi])d\tau
	\label{EqualityFTildeReducedSystem}
	\end{equation}
	and 
	\begin{equation*}
	\mathcal{R}(\varphi) := \partial_{x_1}(D^2\tilde{W}(0)\nabla_h 
	\varphi)_1 + 
	\frac{1}{h}\operatorname{div}_{x'}\big(D^2\tilde{W}(0)^\sim \partial_{x_1}\varphi\otimes
	e_1\big)
	\end{equation*}
	satisfying 
	\begin{equation*}
	\int_S \varphi_{(0)}(x_1, x') dx' =0 \;\;\text{ and }\;\; \int_S \varphi_{(0)}(x_1, x')\cdot x^\perp dx' =0
	\end{equation*}
	for almost all $x_1\in (0, L)$. Then due to the inequality of Lemma \ref{LemmaReducedDimensionHigherRegularity} below it follows
	\begin{align*}
	\frac{1}{h^2}\|\nabla_{x'}^2 \varphi(x_1, \cdot)\|_{L^2(S)} &= \|\nabla_{h, x'}^2 \varphi_{(0)} (x_1, \cdot)\|_{L^2(S)}\\
	&\leq C\bigg(\|\tilde{g}(x_1, \cdot)\|_{L^2(S)} + \frac{1}{h}\|g_N(x_1,\cdot) - a_N(x_1,\cdot)\|_{H^{\frac{1}{2}}(\partial S)}\bigg)
	\end{align*}
	for a.e. $x_1\in (0,L)$. Using the generalised Poincaré's inequality 
	\begin{equation*}
		\|a\|_{L^2(S)} \leq C\bigg(\|\nabla_{x'} a\|_{L^2(S)} + \bigg|\int_{\partial S} a d\sigma(x')\bigg|\bigg)\quad\text{ for all } a\in H^1(S)
	\end{equation*}
	 we obtain with the boundedness of $\tr{\partial S}\colon H^1(S)\to H^\frac{1}{2}(\partial S)$
	\begin{equation*}
		\|q\|_{H^\frac{1}{2}(\partial S)} \leq C\bigg(\|\nabla_{x'}a(q)\|_{L^2(S)} + \bigg|\int_{\partial S} q d\sigma(x')\bigg|\bigg) \quad\text{ for all } q\in H^\frac{1}{2}(\partial S)
	\end{equation*}
	for $a(q)\in H^1(S)$ such that $\operatorname{tr}_{\partial S} (a(q)) = q$ in $H^\frac{1}{2}(\partial S)$. Such an $a(q)$ exists using a classical extension operator $E\colon H^\frac{1}{2}(\partial S) \to H^1(S)$, which is right inverse to $\operatorname{tr}_{\partial S}$. Applying the preceding inequality on $g_N - a_N$  we deduce 
	\begin{align*}
		\|\nabla_{h, x'}^2 \varphi_{(0)} (x_1, \cdot)\|_{L^2(S)} &\leq C\bigg(\|\tilde{g}(x_1,\cdot)\|_{L^2(S)} + \frac{1}{h}\|\nabla_{x'}(a(g_N)(x_1, \cdot) - r_N (x_1, \cdot))\|_{L^2(S)}\\
		&\qquad + \frac{1}{h}\bigg|\int_{\partial S} (g_N - a_N)(x_1,x') d\sigma(x')\bigg|\bigg).
	\end{align*}
	Using Gauss' theorem and \eqref{reducedSystem} leads to
	\begin{align*}
		\int_{\partial S} (g_N - a_N)(x_1,x') d\sigma(x') = h\int_S \tilde{g}(x_1, x') dx'.
	\end{align*}
	Integration with respect to $x_1$ yields
	\begin{equation*}
	\|\nabla_{h, x'}^2 \varphi\|_{L^2(\Omega)} \leq C\bigg(\|\tilde{g}\|_{L^2(\Omega)} + \frac{1}{h}\|\nabla_{x'} a(g_N)\|_{L^2(\Omega)} + \frac{1}{h}\|\nabla_{x'} r_N\|_{L^2(\Omega)}\bigg).
	\end{equation*}
	We can now estimate each term separately, beginning with $\tilde{g}$. As $\mathcal{R}(\varphi)$ is a linear combination of terms involving only $\partial_{x_1} \nabla_h \varphi$ it follows
	\begin{equation}
		\|\mathcal{R}(\varphi)\|_{L^2(\Omega)} \leq C\|\partial_{x_1} \nabla_h \varphi\|_{L^2(\Omega)} \leq  C\bigg\|\frac{1}{h} \e_h(\partial_{x_1}\varphi)\bigg\|_{L^2(\Omega)},
		\label{Equation_RofPhiIsL2}
	\end{equation}
	where we used Korn's inequality and the fact that $\varphi$ is $L$-periodic in $x_1$ direction. Next we have 
	\begin{align}
	\int_0^1 
	\operatorname{div}_h(D^3\tilde{W}(\tau \nabla_h u_h)[\nabla_h u_h, \nabla_h 
	\varphi])d\tau &= \operatorname{div}_h\bigg(\int_0^1 D^3\tilde{W}(\tau \nabla_h u_h)d\tau[\nabla_h u_h, \nabla_h \varphi]\bigg)\notag\\
	& = \operatorname{div}_h \Big(G(\nabla_h u_h) [\nabla_h u_h, \nabla_h \varphi] \Big), \label{localDefinitionOfG}
	\end{align}
	where $G\in C^\infty(\overline{B_\e(0)}, \mathcal{L}^3(\R^{3\times 3}))$ for some suitable $\e >0$. Thus it follows with the identification of $\mathcal{L}^1(\Rtimes)$ with $\Rtimes$ via the standard scalar product
	\begin{align*}
	\|\divh(G(\nabla_h u_h) [\nabla_h u_h, \nabla_h \varphi])\|_{L^2(\Omega)}  &\leq \frac{1}{h} \sup_{|\alpha| = 1}\|\partial_x^\alpha (G(\nabla_h u_h)[\nabla_h u_h, \nabla_h \varphi])\|_{L^2(\Omega; \mathcal{L}^1(\Rtimes))}\\
	&\leq \frac{1}{h} \sup_{|\alpha| = 1} \Big(\|G(\nabla_h u_h)[\nabla_h u_n, \nabla_h \partial_x^\alpha \varphi]\|_{L^2(\Omega; \L^1(\Rtimes))}\\
	& \quad + \|G(\nabla_h u_h)[\nabla_h\partial_x^\alpha u_n, \nabla_h \varphi]\|_{L^2(\Omega; \L^1(\Rtimes))}\\
	& \quad + \|DG(\nabla_h u_h)[\partial_x^\alpha \nabla_h u_h, \nabla_h u_h, \nabla_h \varphi]\|_{L^2(\Omega;\L^1(\Rtimes))}\Big).
	\end{align*}
	As $G\in C^\infty(\overline{B_{\e}(0)}; \L^3(\Rtimes))$ and $\nabla_h u_h \in C^0(\overline{\Omega};\R^3)$ it follows 
	\begin{align*}
		\|G(\nabla_h u_h)[\nabla_h\partial_x^\alpha u_n, \nabla_h \varphi]\|_{L^2(\Omega; \L^1(\Rtimes))} &= \sup_ {\substack{B\in\Rtimes\\ |B|\leq 1}} \bigg(\int_\Omega |G(\nabla_h u_h)[\nabla_h \partial_x^\alpha u_h, \nabla_h \varphi, B]|^2 dx\bigg)^{\frac{1}{2}}\\
		&\leq C \|\nabla_h \partial_x^\alpha u_h\|_{L^4(\Omega)} \|\nabla_h \varphi\|_{L^4(\Omega)} \leq CRh \|\nabla_h \varphi\|_{H^1(\Omega)},
	\end{align*} 
	where we used Hölder inequality, the embedding $H^1(\Omega) \hookrightarrow L^4(\Omega)$ and $\|\nabla_h \partial_x^\alpha u_h\|_{H^1(\Omega)} \leq CRh$ due to \eqref{AssumptionsOnU1}. Analogously using $H^2(\Omega)\hookrightarrow C^0(\overline{\Omega})$ and $\|\nabla_h u_h\|_{H^2(\Omega)} \leq CRh$ it follows
	\begin{equation*}
		\|G(\nabla_h u_h) [\nabla_h u_h, \nabla_h\partial_x^\alpha\varphi]\|_{L^2(\Omega, \L^1(\Rtimes))} \leq CRh \|\nabla_h w\|_{H^1(\Omega)}.
	\end{equation*}
	Finally as $DG\in C^\infty(\overline{B_\e(0)};\L^4(\Rtimes))$ is bounded, we obtain 
	\begin{equation*}
		\|DG(\nabla_h u_h)[\nabla_h \partial_x^\alpha u_h, \nabla_h u_h, \nabla_h \varphi] \|_{L^2(\Omega;\L^1(\Rtimes))} \leq CRh \|\nabla_h \varphi\|_{H^1(\Omega)}.
	\end{equation*}
	Altogether we can conclude 
	\begin{equation}
		\|\divh(G(\nabla_h u_h) [\nabla_h u_h, \nabla_h \varphi])\|_{L^2(\Omega)} \leq CR \|\nabla_h\varphi\|_{H^1(\Omega)}.
		\label{localInequalityForG}
	\end{equation}
	From the definition of $r_N$ it follows
	\begin{align*}
	\frac{1}{h}\|\nabla_{x'} &r_N\|_{L^2(\Omega)}\\
	& = \bigg\|\nabla_{h,x'} (D^2\tilde{W}(0)(\partial_{x_1}\varphi\otimes e_1)) + \int_0^1 \nabla_{h,x'}\big(D^3\tilde{W}(\tau \nabla_h u_h)[\nabla_h u_h, \nabla_h \varphi]\big) d\tau\bigg\|_{L^2(\Omega)}\\
	& \leq \|\nabla_{h,x'} (D^2\tilde{W}(0) (\partial_{x_1}\varphi\otimes e_1))\|_{L^2(\Omega)} + \|\nabla_{h,x'}(G(\nabla_h u_h)[\nabla_h u_h, \nabla_h \varphi])\|_{L^2(\Omega)}
	\end{align*}
	The first term on the right hand side is a linear combination of $\nabla_h \partial_{x_1} \varphi$. Hence
	\begin{equation*}
	\|\nabla_{h, x'} D^2\tilde{W}(0) (\partial_{x_1}\varphi\otimes e_1)\|_{L^2(\Omega)} \leq C\|\nabla_h \partial_{x_1}\varphi\|_{L^2(\Omega)} \leq C\bigg\|\frac{1}{h} \e_h(\varphi)\bigg\|_{H^{0,1}(\Omega)}
	\end{equation*}
	The second term can be bounded analogously to \eqref{localInequalityForG}
	\begin{equation*}
		\|\nabla_{h,x'} G(\nabla_h u_h)[\nabla_h u_h, \nabla_h \varphi]\|_{L^2(\Omega)} \leq CR \|\nabla_h \varphi\|_{H^1(\Omega)}.
	\end{equation*}
	Hence as $\|\nabla_h u\|_{H^1(\Omega)} \leq \|(\nabla_h u, \nabla_h^2 u)\|_{L^2(\Omega)}$ we have 
	\begin{align}
		\|\nabla_{h, x'}^2 \varphi\|_{L^2(\Omega)} & \leq C\bigg(h^2 \|g\|_{L^2(\Omega)} + \bigg\|\frac{1}{h}g_N\bigg\|_{L^2(0,L;H^\frac{1}{2}(\partial S))} + \bigg\|\frac{1}{h} \e_h(\varphi)\bigg\|_{H^{0,1}(\Omega)}\notag\\
		&\qquad + R\|(\nabla_h u, \nabla_h^2 u)\|_{L^2(\Omega)} \bigg). \label{FundamentalInequality_RegularityTheorem}
	\end{align}
	Due to the structure 
	\begin{equation}
	\nabla_h^2 \varphi_j =
	\left(\begin{matrix}
	\partial_{x_1}^2 \varphi_j & (\nabla_{h,x'}\partial_{x_1} \varphi_j)^T\\
	\nabla_{h, x'}\partial_{x_1} \varphi_j & \nabla_{h, x'}^2 \varphi_j
	\end{matrix}\right)
	\label{StructureOfNablaZweiW}
	\end{equation}
	for $j\in \{1,2,3\}$, it follows
	\begin{align*}
		\|\nabla_h^2 \varphi\|_{L^2(\Omega)} \leq C\bigg(\bigg\|\frac{1}{h}\e_h(\varphi)\bigg\|_{L^2(\Omega)} &+ \bigg|\frac{1}{h}\int_\Omega \varphi\cdot x^\perp dx\bigg|\bigg) \\
		& + \bigg\|\frac{1}{h}\e_h(\varphi)\bigg\|_{H^{0,1}(\Omega)} + \|\nabla_{h, x'}^2 \varphi\|_{L^2(\Omega)}.
	\end{align*}
	Hence, by plugging all inequalities into \eqref{FundamentalInequality_RegularityTheorem} and applying Korn's inequality
	\begin{align*}
		\|\nabla_{h, x'}^2 \varphi\|_{L^2(\Omega)} &\leq C\bigg(h^2\|g\|_{L^2(\Omega)} + \bigg\|\frac{1}{h}\e_h(\varphi)\bigg\|_{H^{0,1}(\Omega)} + \bigg\|\frac{1}{h} g_N\bigg\|_{L^2(0,L;H^\frac{1}{2}(\partial S))}\\
		& \qquad + R \bigg|\frac{1}{h} \int_\Omega \varphi\cdot x^\perp dx\bigg| \bigg) + CR \|(\nabla_{h,x'}^2 \varphi)\|_{L^2(\Omega)}
	\end{align*}
	Using an absorption argument for $R_0 \in (0,1]$ sufficiently small and the structure in \eqref{StructureOfNablaZweiW} it follows
	\begin{align*}
		\|\nabla_h^2 \varphi\|_{L^2(\Omega)} &\leq \|\partial_{x_1}^2\varphi\|_{L^2(\Omega)} + 2\|\nabla_{h,x'} \partial_{x_1} \varphi\|_{L^2(\Omega)} + \|\nabla_{h,x'}^2 \varphi\|_{L^2(\Omega)}\\
		&\leq  C\bigg(h^2\|g\|_{L^2(\Omega)} + \bigg\|\frac{1}{h}g_N\bigg\|_{L^2(0,L;H^\frac{1}{2}(\partial S))} + \bigg\|\frac{1}{h}\e_h(\varphi)\bigg\|_{H^{0,1}(\Omega)} +  R\bigg|\frac{1}{h} \int_\Omega \varphi\cdot x^\perp dx\bigg|\bigg)
	\end{align*}
	as 
	\begin{equation*}
		\|\nabla_{h,x'} \partial_{x_1}\varphi\|_{L^2(\Omega)} \leq \|\nabla_h \partial_{x_1}\varphi\|_{L^2(\Omega)} \leq C\bigg\|\frac{1}{h} \e_h(\partial_{x_1} \varphi)\bigg\|_{L^2(\Omega)} \leq C\bigg\|\frac{1}{h} \e_h(\varphi)\bigg\|_{H^{0,1}(\Omega)}
	\end{equation*}
	and $\|\partial_{x_1}^2 \varphi\|_{L^2(\Omega)} \leq \left\|\frac{1}{h} \e_h(\varphi)\right\|_{H^{0,1}(\Omega)}$ holds. Finally, \eqref{InequalityHigherRegularity} for $k=0$ follows from 
	\begin{equation*}
		\bigg\|\nabla \frac{1}{h}\e_h(\varphi)\bigg\|_{L^2(\Omega)} \leq C\bigg( \|\nabla_h^2 \varphi\|_{L^2(\Omega)} + \bigg\|\frac{1}{h} \e_h(\varphi)\bigg\|_{H^{0,1}(\Omega)}\bigg).
	\end{equation*}
	In the case of $k=1$ we prove \eqref{InequalityHigherRegularity} in two steps. The first step consists in differentiating \eqref{LocalStationaryNonZeroNeumannBoundaryEquationHigherRegularity} in direction of $x_1$ and with \eqref{InequalityHigherRegularity} for $k=0$ we obtain
		\begin{align*}
			\|\nabla_h^2 \partial_{x_1} \psi\|_{L^2(\Omega)} \leq C\bigg(h^2&\|q\|_{H^{0,1}(\Omega)} + \bigg\|\frac{1}{h} \e_h(\psi)\bigg\|_{H^{0,2}(\Omega)} + \bigg\|\frac{1}{h}q_N\bigg\|_{H^1(0,L;H^\frac{1}{2}(\partial S))}\\
			& + R \bigg|\frac{1}{h} \int_\Omega \psi\cdot x^\perp dx\bigg| + Rh \|\nabla_h^2 \psi\|_{H^1(\Omega)}\bigg).
		\end{align*}
		In the second step we apply classical higher regularity theory similarly to Lemma \ref{LemmaReducedDimensionHigherRegularity} below, see for instance \cite[Theorem 4.18]{McLean}, and analogous inequalities as in the first case.
\end{proof}
\begin{lem}
	Let the assumptions of Theorem \ref{TheoremHigherRegularity} be satisfied and $\varphi_{(0)}$ be as in \eqref{localDefitionSolutionReducedSystem} a solution of \eqref{reducedSystem}. Then for almost all $x_1\in (0,L)$ it holds
	\begin{equation}
	\frac{1}{h^2} \|\varphi_{(0)}\|_{H^2(S)} \leq C\bigg(\|\tilde{g}(x_1, \cdot)\|_{L^2(S)} + \frac{1}{h}\|(g_N - a_N)(x_1,\cdot)\|_{H^{\frac{1}{2}}(\partial S)}\bigg)
	\label{ReducedDimensionHigherRegularityInequality}
	\end{equation}
	for some $C>0$ independent of $\varphi, \tilde{g}, g_N, a_N$ and $x_1$.
	\label{LemmaReducedDimensionHigherRegularity}
\end{lem}
\begin{proof}
	As we have $\varphi_{(0)}(x_1, \cdot) \in H^2(S)$, one can test the equation \eqref{reducedSystem} with $\varphi_{(0)}(x_1, \cdot)$ to obtain
	\begin{equation*}
	\frac{1}{h^2} \Big(D^2\tilde{W}(0)^\approx \nabla_{x'} \varphi_{(0)}, \nabla_{x'} \varphi_{(0)}\Big)_{L^2(S)} = (\tilde{g}, \varphi_{(0)})_{L^2(S)} + \frac{1}{h}(g_N - a_N, \varphi_{(0)})_{L^2(\partial S)}
	\end{equation*}
	for almost all $x_1\in (0,L)$. Using now the Legendre-Hadamard condition, Korn's inequality in two dimensions and Poincaré's inequality it follows, due to the fact that $\varphi_{(0)}$ is mean value free
	\begin{align*}
	&\Big(D^2\tilde{W}(0)^\approx \nabla_{x'} \varphi_{(0)}, \nabla_{x'} \varphi_{(0)}\Big)_{L^2(S)}  =  \Big(D^2\tilde{W}(0) \bigg(\begin{matrix}0\\ \nabla_{x'}\end{matrix}\bigg) \varphi_{(0)}, \bigg(\begin{matrix}0\\ \nabla_{x'}\end{matrix}\bigg)\varphi_{(0)}\Big)_{L^2(S)}\\
	&\quad \geq  c_0 \bigg\|\operatorname{sym}\bigg(\bigg(\begin{matrix}0\\ \nabla_{x'}\end{matrix}\bigg)\varphi_{(0)}\bigg)\bigg\|^2_{L^2(S)}\geq C \|\nabla_{x'} \varphi_{(0)}\|_{L^2(S)}^2 - C \bigg|\int_S \varphi_{(0)}\cdot x^\perp dx'\bigg|^2\\
	&\quad \geq C \|\varphi_{(0)}\|_{H^1(S)}^2.
	\end{align*}
	Here we used that $\int_S \varphi_{(0)}\cdot x^\perp dx' = 0$. Thus applying Young's inequality and an absorption argument we are led to
	\begin{equation*}
	\frac{1}{h^2}\|\varphi_{(0)}\|_{H^1(S)} \leq C\bigg(\|\tilde{g}\|_{L^2(S)} + \frac{1}{h} \|g_N - a_N\|_{H^\frac{1}{2}(\partial S)}\bigg).
	\end{equation*}
	For higher regularity we apply standard results, found in \cite{GiaquintaMartinazzi} or \cite{McLean}.
	When considering the data one notices that $\tilde{g}(x_1, \cdot)\in L^2(S; \R^3)$ and $(g_N - a_N)(x_1, \cdot)\in H^{\frac{1}{2}}(\partial S; \R^3)$ for almost every $x_1\in (0,L)$ holds. This is a consequence on one hand from the assumptions on $g\in L^2(\Omega)$ and $g_N\in L^2(0,L;H^\frac{1}{2}(S))$. On the other hand we know that $\mathcal{R}(\varphi)\in L^2(\Omega)$ due to \eqref{Equation_RofPhiIsL2}. Moreover, because $\tr{\partial S}\colon H^1(S) \to H^\frac{1}{2}(\partial S)$ it follows $a_N\in H^\frac{1}{2}(\partial S)$. The Legendre-Hadamard condition of $D^2\tilde{W}(0)^\approx$ is inherited from $D^2\tilde{W}(0)$. Finally due to Korn's inequality in two dimensions and Poincaré's inequality with mean value we can conclude that $\mathcal{L}$, defined as
	\begin{equation*}
		\mathcal{L} u := -\sum_{\alpha, \beta=1}^{2} \partial_\alpha(B^{\alpha\beta}\partial_\beta u) := -\operatorname{div}_{x'} \Big(D^2\tilde{W}(0)^\approx \nabla_{x'} u\Big)
	\end{equation*}
	 is weakly coercive on $H^1(S)$. Hence we obtain
	\begin{equation*}
	\frac{1}{h^2}\|\varphi_{(0)}\|_{H^2(S;\R^3)} \leq C\bigg(\frac{1}{h^2}\|\varphi_{(0)}\|_{H^1(S)} + \frac{1}{h}\|(g_N - a_N)(x_1,\cdot)\|_{H^{\frac{1}{2}}(\partial S)} + \|\tilde{g}\|_{L^2(S)}\bigg).
	\end{equation*}
	Putting the above inequalities together leads to the desired result.
\end{proof}
In case we have homogeneous Neumann boundary conditions, we can refine Theorem \ref{TheoremHigherRegularity} in the following way.
\begin{Coro}
	Assume $u_h$ fulfils \eqref{AssumptionsOnU1} and $k=0, 1$. If $w\in H^{2+k}_{per}(\Omega)$ satisfies 
	\begin{align}
		\left\{
		\begin{aligned}
		- \frac{1}{h^2} \operatorname{div}_h (D^2\tilde{W}(\nabla_h u_h)\nabla_h w) &= f \qquad\text{in } \Omega\\
		D^2\tilde{W}(\nabla_h u_h)[\nabla_h w]\nu\Big|_{(0,L)\times \partial S} &= 0 \qquad\text{on } \partial\Omega
		\end{aligned}
		\right.
		\label{StationaryZeroNeumannBoundaryEquation}
	\end{align}
	for some $f\in H^k_{per}(\Omega)$, then it holds
	\begin{equation}
		\bigg\| \bigg(\nabla \frac{1}{h}\e_h(w), \nabla_h^2 w\bigg)\bigg\|_{H^k(\Omega)} \leq C\bigg(\|f\|_{H^k(\Omega)} + R\bigg|\frac{1}{h}\int_\Omega w\cdot x^\perp dx\bigg|\bigg)
		\label{InequalityHigherRegularity_improved}
              \end{equation}
              for some $C>0$ independent $w,f,h,t$, and $u_h$.
	\label{CorollaryHigherRegularity_improved}
\end{Coro}
\begin{proof}
	Applying Theorem \ref{TheoremHigherRegularity} with $\varphi := w$, $g:=f$ and $g_N := 0$ we are lead to 
	\begin{align}\nonumber
          &\bigg\|\bigg(\nabla\frac{1}{h}\e_h(w), \nabla_h^2 w\bigg)\bigg\|_{H^k(\Omega)}\\
          &\quad \leq C\bigg(h^2\|f\|_{H^k(\Omega)} + \bigg\|\frac{1}{h}\e_h(w)\bigg\|_{H^{0,1+k}(\Omega)} + R\bigg|\frac{1}{h}\int_\Omega w\cdot x^\perp dx \bigg|\bigg).
		\label{HigherRegularityInequalityForWZeroNeumannBoundaryConditions}
	\end{align}
	Consequently we want to eliminate the second term on the right hand side. For $k=0$ it follows from \eqref{StationaryZeroNeumannBoundaryEquation} that
	\begin{equation}
	\frac{1}{h^2} \Big(D^2\tilde{W}(\nabla_h u_h) \nabla_h w, \nabla_h \varphi\Big)_{L^2(\Omega)} = (f, \varphi)_{L^2(\Omega)}
	\label{localStationaryWeakEquation}
	\end{equation}
	for $\varphi\in H^1_{per}(\Omega)$. Now we want to choose $\varphi = \partial_{x_1}^{2l}w_{(0)}$ where $w_{(0)} := w - \frac{1}{\mu(\Omega)}(w, x^\perp)_{L^2(\Omega)}x^\perp - \frac{1}{|\Omega|}(w, 1)_{L^2(\Omega)}$ and $l=0,1$.
	First we start with $l=0$. Periodicity of $\nabla_h u_h$, $w$ and $w_{(0)}$, a Taylor expansion and creation of a $-\frac{1}{\mu(\Omega)}(w, x^\perp)_{L^2(\Omega)}x^\perp$ part lead to
	\begin{align*}
	\frac{1}{h^2}\Big(D^2\tilde{W}(0)\nabla_h w_{(0)}, &\nabla_h w_{(0)}\Big)_{L^2(\Omega)} + \frac{1}{h^2}\int_0^1 \Big(D^3\tilde{W}(\tau \nabla_h u_h)[\nabla_h u_h, \nabla_h w_{(0)}], \nabla_h w_{(0)} \Big)_{L^2(\Omega)} d\tau\\
	& + \frac{1}{h^2} \int_0^1 \Big(D^3\tilde{W}(\tau\nabla_h u_h)\Big[\nabla_h u_h, \nabla_h \frac{1}{\mu(\Omega)}(w, x^\perp)_{L^2} x^\perp\Big], \nabla_h w_{(0)}\Big)_{L^2(\Omega)} d\tau\\
	& = (f, w_{(0)})_{L^2(\Omega)}.
	\end{align*}
	Here we used that
	\begin{equation*}
	\Big(D^2\tilde{W}(0) \nabla_h x^\perp, \nabla_h w_{(0)} \Big)_{L^2(\Omega)} = 0
	\end{equation*}
	due to \eqref{globalInequalityCoercivityOfDW}. Moreover, because of the specific construction of $w_{(0)}$ we can deduce  
	\begin{equation*}
	\left\|\frac{1}{h}\e_h(w_{(0)})\right\|_{L^2} = \left\|\frac{1}{h}\e_h(w)\right\|_{L^2} \;\;\text{ and }\;\; \int_\Omega w_{(0)} \cdot x^\perp dx = 0.
	\end{equation*}
	Hence,
	\begin{align*}
	\bigg\|\frac{1}{h}\e_h(w)\bigg\|^2_{L^2(\Omega)} &\leq |(f, w_{(0)})_{L^2(\Omega)}| + \bigg|\frac{1}{h^2}\int_0^1 \Big(D^3\tilde{W}(\tau \nabla_h u_h)[\nabla_h u_h, \nabla_h w_{(0)}], \nabla_h w_{(0)} \Big)_{L^2(\Omega)} d\tau\bigg|\\
	& \quad + \bigg|\frac{1}{h^2} \int_0^1 \Big(D^3\tilde{W}(\tau\nabla_h u_h)\Big[\nabla_h u_h, \nabla_h \frac{1}{\mu(\Omega)}(w, x^\perp)_{L^2} x^\perp\Big], \nabla_h w_{(0)}\Big)_{L^2(\Omega)} d\tau\bigg|.\\
	\end{align*}
	Now, by the Hölder, Young and Poincaré inequality 
	\begin{align*}
	|(f, w_{(0)})_{L^2(\Omega)}| &\leq C(\epsilon)\|f\|^2_{L^2(\Omega)} + \epsilon \|w_{(0)}\|^2_{L^2} \leq C(\epsilon)\|f\|^2_{L^2(\Omega)} + \epsilon \|\nabla  w_{(0)}\|^2_{L^2}\\
	& \leq C(\epsilon)\|f\|^2_{L^2(\Omega)} + \epsilon \bigg\|\frac{1}{h} \e_h(w)\bigg\|^2_{L^2}
	\end{align*}
	for any $\epsilon > 0$. Secondly with Corollary \ref{Corollary_L1BoundsD3TildeW}, \eqref{AssumptionsOnU2} and $\|\nabla_h w_{(0)}\|_{L^2(\Omega)} \leq C\left\|\frac{1}{h}\e_h (w_{(0)})\right\|_{L^2(\Omega)}$, we obtain
	\begin{equation*}
	\bigg|\frac{1}{h^2}\int_0^1 \Big(D^3\tilde{W}(\tau \nabla_h u_h)[\nabla_h u_h, \nabla_h w_{(0)}], \nabla_h w_{(0)} \Big)_{L^2(\Omega)} d\tau\bigg| \leq CR \bigg\|\frac{1}{h}\e_h(w)\bigg\|^2_{L^2(\Omega)}
	\end{equation*}
	and
	\begin{align*}
	&\bigg|\frac{1}{h^2} \int_0^1 \Big(D^3\tilde{W}(\tau\nabla_h u_h)\Big[\nabla_h u_h,\nabla_h \frac{1}{\mu(\Omega)}(w, x^\perp)_{L^2} x^\perp\Big], \nabla_h w_{(0)}\Big)_{L^2(\Omega)} d\tau\bigg| \\
	& \leq \frac{CR}{h} \|(w, x^\perp)_{L^2(\Omega)}\|_{L^2(\Omega)} \|\nabla_h w_{(0)}\|_{L^2(\Omega)} \leq CR \bigg(\bigg\|\frac{1}{h}\e_h(w)\bigg\|^2_{L^2(\Omega)} + \bigg|\frac{1}{h}\int_\Omega w\cdot x^\perp dx\bigg|^2 \bigg).
	\end{align*}
	For sufficiently small $\epsilon > 0$ and $R_0\in (0,1]$ it follows
	\begin{equation}
	\bigg\|\frac{1}{h}\e_h(w)\bigg\|^2_{L^2(\Omega)} \leq C\|f\|^2_{L^2(\Omega)} + CR \bigg|\frac{1}{h}\int_\Omega w\cdot x^\perp dx\bigg|^2.
	\label{LocalSymWInequality}
	\end{equation}
	In order to use $\varphi = \partial_{x_1}^2 w_{(0)}$, we exploit the density of $C^\infty_{per}(\overline{\Omega})$ in $H^1_{per}(\Omega)$. For $\psi\in C^\infty_{per}(\overline{\Omega})$ we can use $\varphi = \partial_{x_1}\psi$ as a testfunction and obtain 
	\begin{equation*}
		\frac{1}{h^2} \Big(D^2\tilde{W}(\nabla_h u_h) \nabla_h w, \nabla_h \partial_{x_1} \psi\Big)_{L^2(\Omega)} = (f, \partial_{x_1} \psi)_{L^2(\Omega)}.
	\end{equation*}
	Integration by parts, the periodicity and the homogeneous Neumann  boundary conditions lead to 
\begin{equation*}
	\frac{1}{h^2} \Big(\partial_{x_1}(D^2\tilde{W}(\nabla_h u_h) \nabla_h w), \nabla_h \psi\Big)_{L^2(\Omega)} = (f, \partial_{x_1} \psi)_{L^2(\Omega)}.
\end{equation*}
The density of $C^\infty_{per}(\overline\Omega)$ implies now that the latter equation holds for $\psi = \partial_{x_1} w_{(0)}\in H^1_{per}(\Omega)$ as well. Thus
\begin{align*}
	\frac{1}{h^2}\Big(D^2\tilde{W}(\nabla_h u_h)\nabla_h\partial_{x_1} &w_{(0)}, \nabla_h \partial_{x_1} w_{(0)}\Big)_{L^2(\Omega)} \\
	&= (f, \partial_{x_1}^2 w_{(0)})_{L^2{(\Omega)}} -  \frac{1}{h^2}\Big((\partial_{x_1} D^2\tilde{W}(\nabla_h u_h))\nabla_h w, \nabla_h \partial_{x_1} w_{(0)}\Big)_{L^2(\Omega)}
\end{align*}
as $\partial_{x_1} w_{(0)} = \partial_{x_1} w$. Using $\int_\Omega \partial_{x_1} w\cdot x^\perp dx = 0$, we can conclude
\begin{align*}
	\bigg\| \frac{1}{h}\e_h(\partial_{x_1} w)\bigg\|_{L^2(\Omega)}^2 &\leq |(f, \partial_{x_1}^2 w)_{L^2(\Omega)}| + \bigg|\frac{1}{h^2}\Big(\partial_{x_1}D^2\tilde{W}(\nabla_h u_h)\nabla_h w, \nabla_h \partial_{x_1} w\Big)_{L^2(\Omega)}\bigg|\\
	&\leq C(\epsilon)\|f\|^2_{L^2(\Omega)} + \epsilon\bigg\|\frac{1}{h} \e_h(\partial_{x_1}w)\bigg\|^2_{L^2(\Omega)} + CR \|\nabla_h w\|_{L^2_h(\Omega)}\|\nabla_h\partial_{x_1} w\|_{L^2_h(\Omega)}.
\end{align*}
In the preceding calculation Hölder's and Young's inequality are used as well as Lemma \ref{AbschätzungenAbleitungD2W}. Hence, by Korn's inequality and \eqref{LocalSymWInequality} 
\begin{align*}
	\bigg\| \frac{1}{h} \e_h(\partial_{x_1} w)\bigg\|^2_{L^2(\Omega)} &\leq C\|f\|^2_{L^2(\Omega)} + CR \Big(\|\nabla_h w\|^2_{L^2_h(\Omega)} + \|\nabla_h \partial_{x_1} w\|^2_{L^2_h(\Omega)}\Big)\\
	& \leq  C\bigg(\|f\|^2_{L^2(\Omega)} + R\bigg\|\frac{1}{h}\e_h (w)\bigg\|^2_{L^2(\Omega)} + R\bigg|\frac{1}{h}\int_\Omega w\cdot x^\perp dx\bigg|^2\bigg)\\
	&\qquad + CR \bigg\|\frac{1}{h}\e_h (\partial_{x_1} w)\bigg\|^2_{L^2(\Omega)}
\end{align*}
Choosing $R_0 \in (0,1]$ sufficiently small leads to
\begin{equation}
	\bigg\| \frac{1}{h} \e_h(\partial_{x_1} w)\bigg\|^2_{L^2(\Omega)} \leq C\bigg(\|f\|^2_{L^2(\Omega)} + R\bigg\|\frac{1}{h}\e_h (w)\bigg\|^2 + R\bigg|\frac{1}{h}\int_\Omega w\cdot x^\perp dx\bigg|^2\bigg).
	\label{LocalSymPartialX1WInequality}
\end{equation}
Combining \eqref{LocalSymWInequality} and \eqref{LocalSymPartialX1WInequality} with \eqref{InequalityHigherRegularity} the case $k=0$ follows.
	
For $k=1$ it remains to estimate $\|\frac{1}{h}\e_h(\partial^2_{x_1} w)\|_{L^2 (\Omega)}$. The bound can be seen as follows: Analogously as above we can choose first $\varphi = \partial^2_{x_1} \psi$ for $\psi\in C^\infty_{per}(\overline{\Omega})$. Thus, twice integration by parts leads to
\begin{equation*}
	\frac{1}{h^2} \Big(\partial_{x_1}^2(D^2\tilde{W}(\nabla_h u_h)\nabla_h w),\nabla_h \psi\Big)_{L^2(\Omega)} = - (\partial_{x_1} f, \partial_{x_1}\psi)_{L^2(\Omega)}.
\end{equation*}
Then the density of $C^\infty_{per}(\overline{\Omega}) \subset H^1_{per}(\Omega)$ implies that the latter equaltiy holds for $\psi = \partial_{x_1}^2 w$. Using
\begin{align*}
	\Big(\partial_{x_1}^2&(D^2\tilde{W}(\nabla_h u_h)\nabla_h w),\nabla_h \partial_{x_1}^2 w\Big)_{L^2(\Omega)} = \Big(D^2\tilde{W}(\nabla_h u_h) \nabla_h \partial_{x_1}^2 w, \nabla_h \partial_{x_1}^2 w\Big)_{L^2(\Omega)}\\
	& + 2 \Big((\partial_{x_1} D^2\tilde{W}(\nabla_h u_h)) \nabla_h \partial_{x_1} w, \nabla_h \partial_{x_1}^2 w\Big)_{L^2(\Omega)} + \Big((\partial_{x_1}^2 D^2\tilde{W}(\nabla_h u_h)) \nabla_h w, \nabla_h \partial_{x_1}^2 w\Big)_{L^2(\Omega)}.
\end{align*}
By virtue of Lemma \ref{AbschätzungenAbleitungD2W}, we obtain
\begin{align*}
	\bigg\| \frac{1}{h}\e_h(\partial_{x_1}^2 w)\bigg\|_{L^2(\Omega)}^2 &\leq |(\partial_{x_1}f, \partial_{x_1}^3 w)_{L^2(\Omega)}| + CR \|\nabla_h\partial_{x_1} w\|_{L^2_h(\Omega)}\|\nabla_h \partial_{x_1}^2 w\|_{L^2_h(\Omega)}\\
	&\qquad + CR \|\nabla_h w\|_{H^1_h(\Omega)}\|\nabla_h \partial_{x_1}^2 w\|_{L^2_h(\Omega)}\\
	&\leq C(\epsilon)\|f\|^2_{H^{0,1}(\Omega)} + \epsilon \bigg\|\frac{1}{h} \e_h(\partial_{x_1}^2 w)\bigg\|_{L^2(\Omega)}^2 +  CR\|\nabla_h\partial_{x_1} w\|_{L^2(\Omega)}^2 \\
	&\qquad + CR \|\nabla_h w\|_{H^1_h(\Omega)}^2 + CR \|\nabla_h \partial_{x_1}^2 w\|_{L^2_h(\Omega)}^2\\
	& \leq C(\epsilon)\|f\|_{H^{0,1}(\Omega)} + (\epsilon + CR)\bigg\|\frac{1}{h} \e_h(\partial_{x_1}^2 w)\bigg\|_{L^2(\Omega)}^2 \\
	&\qquad + CR \bigg\|\frac{1}{h} \e_h(w)\bigg\|_{H^{0,1}(\Omega)} + CR \bigg|\int_\Omega w\cdot x^\perp dx\bigg|
\end{align*}
Finally, choosing $\epsilon$ and $R_0$ small, using an absorption argument and applying \eqref{LocalSymWInequality} and \eqref{LocalSymPartialX1WInequality}, leads to the desired result. 
\end{proof}
Before we use the preceding Corollary \ref{CorollaryHigherRegularity_improved} to obtain higher regularity bounds for \eqref{linearEQ_1}--\eqref{linearEQ_4}, we state standard first order energy estimate for the solution of the dynamic system.
\begin{lem}
	Let $0 < T < \infty$, $h\in (0,1]$, $0 < R \leq R_0$ be given, where $R_0$ is chosen small enough, but independent of $h$. Furthermore, assume that $u_h$ satisfies \eqref{AssumptionsOnU1}.
	For every $f\in L^1(0,T; L^2(\Omega))$, $w_0\in H^1_{per}(\Omega)$ and $w_1\in L^2(\Omega)$ there exists a unique solution $w\in C^0([0,T]; H^1_{per}(\Omega))\cap C^1([0,T]; L^2(\Omega))$ of the system \eqref{linearEQ_1}--\eqref{linearEQ_4} satisfying
	\begin{align}
	\bigg\|\bigg(\partial_t w, \frac{1}{h}\e_h(w)\bigg)\bigg\|_{C^0([0,T];L^2)}^2 \leq C_L \bigg(\|w_1\|_{L^2(\Omega)}^2& + |A_0| + \|f\|^2_{L^1(0,T; L^2(\Omega))} \label{BasicInequality}\\
	&+ (1+T)R \bigg\|\frac{1}{h}\int_\Omega w\cdot x^\perp dx\bigg\|_{C^0([0,T])}^2\bigg)\notag
	\end{align}
	where $C_L>0$ is independent of $h, w, w_0, w_1, f, A_0$ and 
	\[A_0 = \frac{1}{h^2}\Big(D^2\tilde{W}(\nabla_h u_{h}|_{t=0})\nabla_h w_0, \nabla_h w_0\Big)_{L^2(\Omega)}.\]
	\label{Lemma_BasicInequality}
\end{lem}
\begin{proof}
	The proof uses classical energy estimates and can be found in \cite{AmeismeierDiss}, Lemma 5.2.3.
\end{proof}
The second order regularity bounds are given in the following theorem.
\begin{theorem}
	Let $0 < T < \infty$, $h\in (0,1]$, $0 < R \leq R_0$ be given, where $R_0$ is chosen small enough, but independent of $h$. Furthermore, let $u_h$ satisfy \eqref{AssumptionsOnU1}.
	Assume $w\in  \bigcap_{k=0}^2 C^k([0,T]; H^{2-k}(\Omega))$ is the unique solution of the system \eqref{linearEQ_1}--\eqref{linearEQ_4} for some $f\in W^1_1(0,T; L^2(\Omega))$, $w_0\in H^2_{per}(\Omega)$ and $w_1\in H^1_{per}(\Omega)$, then there exist constants $C_{L1}$, $C_1>0$ (independent of $h, T, R, w, f, A_k$) such that
	\begin{align*}
	\bigg\|\bigg(\partial_t^2 w, \nabla_{x,t}\frac{1}{h}\e_h(w)&, \nabla_h^2 w\bigg)\bigg\|^2_{C^0([0,T];L^2(\Omega))} \leq C_{L1} e^{C_1 T R } \bigg(\|f\|^2_{W^1_
		1(0,T; L^2)} + \|(w_1, w_2, f|_{t=0})\|^2_{L^2}\\
	& \qquad + |(A_0, A_1)| + (1+T)R \max_{\sigma = 0,1}  \bigg\|\frac{1}{h}\int_\Omega \partial_t^\sigma w \cdot x^\perp dx\bigg\|_{C^0(0,T)}^2\bigg) 
	\end{align*}
	holds, where
	\begin{equation*}
	A_k := \frac{1}{h^2}\big(D^2\tilde{W}(\nabla_h u_h|_{t=0})\nabla_h w_k, \nabla_h w_k\big)_{L^2(\Omega)} \quad \text{ for } k=0, 1
	\end{equation*}
	and
	\begin{equation*}
	w_2 := \frac{1}{h^2}\divh\big(D^2\tilde{W}(\nabla_h u_h|_{t=0})\nabla_h w_0\big) + f|_{t=0}.
	\end{equation*}
	\label{TheoremSecondOrderInequality}
\end{theorem}
\begin{proof}
	Differentiating \eqref{linearEQ_1} with respect to $t$ and testing the result with $\partial_t^2 w$ yields
	\begin{align}
	\begin{split}
	\Big(\partial_t^3 w, \partial_t^2 w \Big)_{L^2(\Omega)}& + \frac{1}{h^2}\Big(D^2\tilde{W}(\nabla_h u_h)\nabla_h \partial_t w,\nabla_h \partial_t^2 w \Big)_{L^2(\Omega)}\\
	&= \Big(\partial_t f, \partial_t^2 w \Big)_{L^2(\Omega)} - \frac{1}{h^2}\Big(\partial_t D^2\tilde{W}(\nabla_h u_h)\nabla_h w, \nabla_h \partial_t^2 w \Big)_{L^2(\Omega)}.
	\end{split}	
	\end{align}	
	As
	\begin{align*}
		\frac{d}{dt} \Big(D^2\tilde{W}(\nabla_h u_h)\nabla_h \partial_t w, \nabla_h \partial_t w \Big)_{L^2(\Omega)} &= 2\Big(D^2\tilde{W}(\nabla_h u_h)\nabla_h \partial_t w, \nabla_h \partial_t^2 w \Big)_{L^2(\Omega)} \\
		&\quad + \Big(\partial_t D^2\tilde{W}(\nabla_h u_h) \nabla_h \partial_t w, \nabla_h \partial_t w \Big)_{L^2(\Omega)}
	\end{align*}
	and
	\begin{align*}
		\frac{d}{dt} \Big(\partial_t &D^2\tilde{W}(\nabla_h u_h) \nabla_h w, \nabla_h \partial_t w \Big)_{L^2(\Omega)} = \Big(\partial_t D^2\tilde{W}(\nabla_h u_h)\nabla_h w, \nabla_h \partial_t^2 w \Big)_{L^2(\Omega)}\\
		& + \Big(\partial_t D^2\tilde{W}(\nabla_h u_h)\nabla_h \partial_t w, \nabla_h \partial_t w \Big)_{L^2(\Omega)} + \Big(\partial_t^2 D^2\tilde{W}(\nabla_h u_h)\nabla_h w, \nabla_h \partial_t w \Big)_{L^2(\Omega)}
	\end{align*}
	it follows
	\begin{align*}
	\frac{d}{dt}&\frac{1}{2} \bigg[\|\partial_t^2 w\|_{L^2(\Omega)}^2 + \frac{1}{h^2}\Big(D^2\tilde{W}(\nabla_h u_h)\nabla_h \partial_t w, \nabla_h \partial_t w\Big)_{L^2(\Omega)}  \bigg] \\
	& \leq |(\partial_t f, \partial_t^2 w)_{L^2(\Omega)}| + \frac{3}{2}\bigg|\frac{1}{h^2}\Big(\partial_t D^2\tilde{W}(\nabla_h u_h)\nabla_h \partial_t w, \nabla_h \partial_t w\Big)_{L^2(\Omega)}\bigg|\\
	& \quad +\bigg|\frac{1}{h^2} \Big(\partial_t^2 D^2\tilde{W}(\nabla_h u_h) \nabla_h w, \nabla_h \partial_t w\Big)_{L^2(\Omega)}\bigg| -\frac{d}{dt}\frac{1}{h^2}\Big(\partial_t D^2\tilde{W}(\nabla_h u_h)\nabla_h w, \nabla_h \partial_t w\Big)_{L^2(\Omega)}.
	\end{align*}
	Due to \eqref{equationAbschätzungenAbleitungD2W} and Korn's inequality
	\begin{align*}
	\bigg|\frac{1}{h^2}\Big(\partial_t D^2\tilde{W}(\nabla_h u_h)\nabla_h \partial_t w, \nabla_h \partial_t w\Big)_{L^2(\Omega)}\bigg| &\leq CR \|\nabla_h \partial_t w\|_{L_h^2(\Omega)}^2 \\
	& \leq CR\bigg(\bigg\|\frac{1}{h}\e_h(\partial_t w)\bigg\|_{L^2(\Omega)}^2 + \frac{1}{h^2} \bigg|\int_\Omega \partial_t w\cdot x^\perp\bigg|^2 \bigg),\\
	\bigg|\frac{1}{h^2} \Big(\partial_t^2 D^2\tilde{W}(\nabla_h u_h) \nabla_h w, \nabla_h \partial_t w\Big)_{L^2(\Omega)}\bigg| &\leq CR \|\nabla_h w\|_{H^1_h(\Omega)} \|\nabla_h \partial_t w\|_{L_h^2(\Omega)}\\
	&\leq CR \|\nabla_h w\|_{H^1_h(\Omega)}^2 + CR\|\nabla_h \partial_t w\|_{L_h^2(\Omega)}^2.
	\end{align*}
	Using the definition of $H^1_h(\Omega)$ and Korn inequality, it follows
	\begin{align*}
	\|\nabla_h w\|_{H^1_h(\Omega)}^2 &\leq C\|\nabla_h w\|_{L^2_h(\Omega)}^2 + C\sum_{i=1}^{3} \|\partial_{x_i} \nabla_h w\|_{L^2_h(\Omega)}^2\\
	&\leq C\bigg\|\frac{1}{h}\e_h(w)\bigg\|_{L^2(\Omega)}^2 + \frac{C}{h^2}\bigg|\int_\Omega w\cdot x^\perp dx\bigg|^2 + C\bigg\|\bigg(\nabla \frac{1}{h}\e_h(w), \nabla_h^2 w\bigg)\bigg\|_{L^2(\Omega)}^2.
	\end{align*}
	Moreover,
	\begin{align*}
	\sup_{\tau\in [0, t]} \bigg|\frac{1}{h^2} \Big(\partial_t D^2\tilde{W}&(\nabla_h u_h(\tau))\nabla_h w(\tau)), \nabla_h \partial_t w(\tau)\Big)_{L^2(\Omega)}\bigg|\\
	&\leq CR \sup_{\tau\in [0, t]}  \Big(\|\nabla_h w(\tau)\|_{L^2_h(\Omega)} \|\nabla_h \partial_t w(\tau)\|_{L^2_h(\Omega)}\Big)\\
	& \leq CR \bigg(\bigg\|\frac{1}{h}\e_h(w)\bigg\|_{L^\infty(0,t; L^2(\Omega))}^2 + \bigg\|\frac{1}{h}\int_\Omega w\cdot x^\perp dx\bigg\|_{L^\infty(0,t)}^2\bigg)\\
	& \quad + CR \bigg(\bigg\|\frac{1}{h}\e_h(\partial_t w)\bigg\|_{L^\infty(0,t; L^2(\Omega))}^2 + \bigg\|\frac{1}{h}\int_\Omega \partial_t w\cdot x^\perp dx\bigg\|_{L^\infty(0,t)}^2\bigg).
	\end{align*}
	Putting everything together, using the coercivity \eqref{Coercivity} of $D^2\tilde{W}(\nabla_h u_h)$ and Young's inequality we obtain
	\begin{align*}
	\sup_{\tau\in [0, t]} & \bigg\|\Big(\partial_t^2 w(\tau), \frac{1}{h}\e_h(\partial_t w(\tau))\Big)\bigg\|_{L^2(\Omega)}^2\\
	&\leq \|w_2\|_{L^2(\Omega)}^2 + |A_1| + CR \sup_{\tau\in [0, t]} \bigg|\frac{1}{h} \int_\Omega \partial_t w\cdot x^\perp dx \bigg|^2 + C\|\partial_t f\|_{L^1(0,T;L^2(\Omega))}^2\\
	&\quad + \frac{1}{2}\|\partial_t^2 w\|_{L^\infty(0, t; L^2(\Omega))}^2 + CR \bigg(\bigg\|\frac{1}{h}\e_h(\partial_t w)\bigg\|_{L^2(0,t; L^2(\Omega))}^2 + \bigg\|\frac{1}{h}\int_\Omega \partial_t w\cdot x^\perp dx\bigg\|^2_{L^2(0,t)}\bigg)\\
	&\quad + CR\bigg(\bigg\|\frac{1}{h}\e_h(w)\bigg\|_{L^2(0,t; L^2(\Omega))}^2 + \bigg\|\frac{1}{h}\int_\Omega w\cdot x^\perp dx\bigg\|_{L^2(0,t)}^2\bigg)\\
	& \quad + CR \bigg(\bigg\|\frac{1}{h}\e_h(w)\bigg\|_{L^\infty(0,t; L^2(\Omega))}^2 + \bigg\|\frac{1}{h}\int_\Omega w\cdot x^\perp dx\bigg\|_{L^\infty(0,t)}^2\bigg)\\
	&\quad + CR \bigg(\bigg\|\frac{1}{h}\e_h(\partial_t w)\bigg\|_{L^\infty(0,t; L^2(\Omega))}^2 + \bigg\|\frac{1}{h}\int_\Omega \partial_t w\cdot x^\perp dx\bigg\|_{L^\infty(0,t)}^2\bigg)\\
	&\quad + CR\bigg\|\Big(\nabla \frac{1}{h}\e_h(w), \nabla_h^2 w\Big)\bigg\|_{L^2(0,t; L^2(\Omega))}^2
	\end{align*}
	uniformly in $0\leq t \leq T$. We use an absorption argument for
	\begin{equation*}
		\frac{1}{2}\|\partial_t^2 w\|_{L^\infty(0, t; L^2(\Omega))}^2 \;\;\text{ and }\;\; CR \Big\|\frac{1}{h}\e_h(\partial_t w)\Big\|_{L^\infty(0,t; L^2(\Omega))}^2
	\end{equation*}
	and the fact that $L^\infty(0,t) \hookrightarrow L^2(0,t)$ with $\|g\|_{L^2(0,t)} \leq \sqrt{t} \|g\|_{L^\infty(0,t)}$. Now, due to \eqref{InequalityHigherRegularity_improved}
	\begin{align*}
	\sup_{\tau\in [0, t]} \bigg\|\bigg(\nabla\frac{1}{h}\e_h (w(\tau)), \nabla_h^2 w(\tau)\bigg)\bigg\|^2_{L^2(\Omega)} &\leq C \|f\|^2_{C^0(0,t; L^2)} + C\|\partial_t^2 w\|^2_{C^0(0,t; L^2(\Omega))}\\
	&\qquad + CRt  \bigg\|\frac{1}{h}\int_\Omega w\cdot x^\perp dx\bigg\|^2_{L^\infty(0,t)}.
	\end{align*}
	Applying 
	\begin{equation*}
	\|g\|_{C^0([0,t]; L^2(\Omega))} \leq C\Big(\|g\|_{W^1_1(0,t;L^2(\Omega))} + \|g(0)\|_{L^2(\Omega)}\Big) \;\text{ for all } g\in W^1_1(0,t;L^2(\Omega))
	\end{equation*}
	for $f$, we arrive at
	\begin{align*}
	\sup_{\tau\in [0, t]} &\bigg\|\bigg(\partial_t^2 w(\tau), \nabla_{x,t}\frac{1}{h}\e_h(w(\tau)), \nabla_h^2 w(\tau)\bigg)\bigg\|^2_{L^2(\Omega)}\\
	&\leq \sup_{\tau\in [0, t]} \bigg\|\bigg(\partial_t^2 w(\tau), \frac{1}{h}\e_h(\partial_t w(\tau))\bigg)\bigg\|_{L^2(\Omega)}^2 + \sup_{\tau\in [0, t]} \bigg\|\bigg(\nabla\frac{1}{h}\e_h (w(\tau)), \nabla_h^2 w(\tau)\bigg)\bigg\|_{L^2(\Omega)} \\
	& \leq \|w_2\|^2_{L^2(\Omega)} + |A_1| + C\|\partial_t f\|^2_{L^1(0, T; L^2)} + C\|f\|^2_{W^1_1(0,t; L^2(\Omega))} + C\|f|^2_{t=0}\|_{L^2(\Omega)}\\
	&\quad + CR(1+t) \bigg\|\frac{1}{h}\e_h(w)\bigg\|_{L^\infty(0,t; L^2(\Omega))}^2 + CR (1+t) \max_{\sigma = 0,1} \bigg\|\frac{1}{h} \int_\Omega \partial_t^\sigma w\cdot x^\perp dx\bigg\|^2_{L^\infty(0,t)}\\
	&\quad+ CR \bigg\|\bigg(\frac{1}{h}\e_h(\partial_t w), \nabla \frac{1}{h}\e_h(w), \nabla_h^2 w\bigg)\bigg\|^2_{L^2(0,t; L^2(\Omega))}
	\end{align*}
	where we used $L^\infty(0,t)\hookrightarrow L^2(0,t)$ again. Hence, due \eqref{BasicInequality} and choosing $R_0$ sufficiently small
	\begin{align*}
	\sup_{\tau\in [0, t]}& \bigg\|\bigg(\partial_t^2 w(\tau), \nabla_{x,t}\frac{1}{h}\e_h(w(\tau)), \nabla_h^2 w(\tau)\bigg)\bigg\|^2_{L^2(\Omega)}\\
	& \leq \|w_2\|^2_{L^2(\Omega)} + |A_1| + C\|f|_{t=0}\|^2_{L^2(\Omega)} + CR\lambda_t \Big(\|w_1\|^2_{L^2} + |A_0| + \|f\|^2_{L^1(0,t; L^2)}\Big)\\
	&\quad + C\|\partial_t f\|^2_{L^1(0, T; L^2)} + C\|f\|^2_{W^1_1(0,t; L^2)} + CR (1+t) \max_{\sigma = 0,1} \bigg\|\frac{1}{h} \int_\Omega \partial_t^\sigma w\cdot x^\perp dx\bigg\|^2_{L^\infty(0,t)}\\
	&\quad + CR \bigg\|\bigg(\partial_t^2 w, \nabla_{t,x}\frac{1}{h}\e_h(w), \nabla_h^2 w\bigg)\bigg\|^2_{L^2(0,t; L^2)},
	\end{align*}
	where $\lambda_t = \max\{1, t\}$. As $0 < t < T < \infty$, there exists $R_0\in (0,1]$, such that
	\begin{equation*}
	C\|\partial_t f\|^2_{L^1(0, T; L^2)} + C\|f\|^2_{W^1_1(0,t; L^2)} + CR\lambda_t \|f\|^2_{L^1(0,t; L^2)} \leq C \|f\|^2_{W^1_1(0,T; L^2)}
	\end{equation*}
	because of $W^1_1(0,t)\hookrightarrow L^1(0,t)$ with $\|g\|_{L^1(0,t)} \leq \|g\|_{W^1_1(0,t)}$. The Lemma of Gronwall yields then 
	\begin{align*}
	\bigg\|\bigg(\partial_t^2 w, &\nabla_{x,t}\frac{1}{h}\e_h(w), \nabla_h^2 w\bigg)\bigg\|^2_{C([0,T];L^2(\Omega))} \leq C_{L1} e^{C_1 (1+T)R} \bigg(\|f\|^2_{W^1_1(0,T; L^2)} + |(A_0, A_1)| \\
	&\qquad + \|(w_0, w_1, f|_{t=0})\|^2_{L^2} + (1+T)R \max_{\sigma = 0,1} \bigg\|\frac{1}{h} \int_\Omega \partial_t^\sigma w\cdot x^\perp dx\bigg\|^2_{C^0([0,T])}\bigg). \qedhere
	\end{align*}
\end{proof}
\begin{bem}
	The existence of a unique solution $w\in C^0([0,T]; H^1_{per}(\Omega))\cap C^1([0,T]; L^2(\Omega))$ for \eqref{linearEQ_1}--\eqref{linearEQ_4} under the conditions of Theorem \ref{TheoremSecondOrderInequality} follows from classical PDE theory. For higher regularity one would then apply hyperbolic regularity theory, cf. \cite{WlokaPDE} and \cite[Chapter 5]{LionsMagenes}. Necessarily we need at this point that suitable compatibility conditions hold. As in the proof of the main result solutions of the linearised system are obtained via differentiation, we will not show the details here. 
\end{bem}
In order to prove the main theorem via contradiction, we need that $\nabla_{x,t}^2 u_h$ is regular enough, to avoid blow ups. In our situation we need the following regularity theorem.
\begin{theorem}
	Let $0 < T < \infty$, $h\in (0,1]$, $0 < R \leq R_0$ be given, where $R_0$ is chosen small enough, but independent of $h$ and let $u_h$ satisfy \eqref{AssumptionsOnU1}. Assume $w\in  \bigcap_{k=0}^3 C^k([0,T]; H^{3-k}_{per}(\Omega))$ to be the unique solution of the linearised system \eqref{linearEQ_1}--\eqref{linearEQ_4} for some $f\in W^2_1(0,T; L^2(\Omega))\cap W^1_1(0,T;H^1_{per}(\Omega))$, $w_0\in H^3_{per}(\Omega)$ and $w_1\in H^2_{per}(\Omega)$. Then there exist constants $C_{max}\geq 1$, $C' >0$ depending on $T$ such that
	\begin{align*}
	&\max_{\sigma = 0,1,2}\bigg\|\bigg(\partial_t^{1+\sigma} w, \nabla_{t,x}^\sigma\frac{1}{h}\e_h(w), \nabla_{x,t}^{\max\{\sigma-1, 0\}}\nabla_h^2 w\bigg)\bigg\|^2_{C^0([0,T];L^2(\Omega))}\\
	&\;\leq C_{max} e^{C'(1+T) R} \bigg(\|\partial_t^2 f\|^2_{L^1(0,T;L^2)} + \|\partial_t f\|^2_{L^\infty(0,T;L^2)\cap L^1(0,T;H^{0,1})} + \|f\|^2_{L^\infty(0,T;H^1)} \\
	& \quad + \|(w_1, w_2, w_3, f|_{t=0})\|^2_{L^2} + \bigg\|\bigg(\frac{1}{h}\e_h(\partial_{x_1}w_1), \partial_{x_1} w_2\bigg)\bigg\|^2_{L^2}\\
	& \quad +\max_{k=0,1,2}\bigg(\bigg\|\frac{1}{h}\e_h(w_k)\bigg\|^2_{L^2} + \bigg|\frac{1}{h}\int_\Omega w_k \cdot x^\perp dx\bigg|^2\bigg)\\
	&\quad + \frac{1}{h^2}\bigg\|\int_\Omega \partial_t^2 w\cdot x^\perp  dx\bigg\|^2_{L^\infty(0,T)} + (1+T)R \max_{\sigma = 0,1,2} \bigg\|\frac{1}{h} \int_\Omega\partial_t^\sigma w\cdot x^\perp dx\bigg\|^2_{C^0([0,T])}\bigg)
	\end{align*}
	\label{UniversalInequality}
\end{theorem}
\begin{proof}
	Due to Lemma \ref{Lemma_BasicInequality} and Theorem \ref{TheoremSecondOrderInequality} it remains to bound only the third order terms. 
	
	Differentiation with respect to $z_j$ of the system \eqref{linearEQ_1}--\eqref{linearEQ_2} leads to
	\begin{align*}
		\partial_t^2 \tilde{w}_j - \frac{1}{h^2} \divh\Big(D^2\tilde{W}(\nabla_h u_h)\nabla_h \tw_j\Big) & = \partial_{z_j} f + \frac{1}{h^2}\divh\Big(\partial_{z_j}D^2\tilde{W}(\nabla_h u_h)\nabla_h w\Big),\\
		D^2\tilde{W}(\nabla_h u_h) \nabla_h \tw_j \nu\Big|_{(0,L)\times \partial S} & = -\partial_{z_j} D^2\tilde{W}(\nabla_h u_h) \nabla_h w \nu\Big|_{(0,L)\times\partial S},\\
		\tw \text{ is $L$-periodic}& \text{ with respect to } x_1,
	\end{align*}
	where $\tw_j = \partial_{z_j} w$. First we want to apply Theorem \ref{TheoremHigherRegularity} with $\varphi := \tw_j$ and
	\begin{align*}
		& g:=  \tilde{f}_j + \frac{1}{h^2}\divh\Big(\partial_{z_j}D^2\tilde{W}(\nabla_h u_h)\nabla_h w\Big) - \partial_t^2 \tilde{w}_j,\\
		& g_N := -\partial_{z_j} D^2\tilde{W}(\nabla_h u_h)\nabla_h w \nu\Big|_{(0,L)\times\partial S},
	\end{align*}
	where we used the convention $\tilde{f}_j := \partial_{z_j} f$. Then we obtain
	\begin{align*}
		\bigg\|\bigg(\nabla\frac{1}{h}\e_h(\tw_j), \nabla_h^2 \tw_j \bigg)\bigg\|_{L^2(\Omega)} & \leq Ch^2\bigg\|\tilde{f}_j -\partial_t^2 \tw_j + \frac{1}{h^2}\divh\Big(\partial_{z_j}D^2\tilde{W}(\nabla_h u_h)\nabla_h w\Big)\bigg\|_{L^2(\Omega)} \\
		& \qquad + C \bigg\|\frac{1}{h}\tr{\partial S}\Big(\partial_{z_j} D^2\tilde{W}(\nabla_h u_h) \partial_h w\Big)\bigg\|_{L^2(0,L;H^\frac{1}{2}(\partial S))}\\
		& \qquad + C \bigg\|\frac{1}{h}\e_h(\tw_j)\bigg\|_{H^{0,1}(\Omega)} + CR \bigg|\frac{1}{h}\int_\Omega \tw_j \cdot x^\perp dx\bigg|.
	\end{align*}
	Moreover, as $D^3\tilde{W}(\nabla_h u_h)$ is uniformly bounded and $u_h$ satisfies \eqref{AssumptionsOnU1}
	\begin{align*}
		\Big\|\divh\Big(\partial_{z_j} D^2\tilde{W}(\nabla_h u_h)\nabla_h w\Big)\Big\|_{L^2(\Omega)} &\leq \Big\|\divh\Big(D^3\tilde{W}(\nabla_h u_h)[\nabla_h \partial_{z_j} u_h]\nabla_h w\Big)\Big\|_{L^2(\Omega)}\\
		&\leq CR \Big\|\Big(\nabla_h w,\nabla_h^2 w\Big)\Big\|_{L^2(\Omega)}
	\end{align*}
	and
	\begin{align*}
		\bigg\|\frac{1}{h}\tr{\partial S}\Big(\partial_{z_j} D^2\tilde{W}(\nabla_h u_h) \nabla_h w\Big)\bigg\|_{L^2(0,L;H^\frac{1}{2}(\partial S))} & \leq C\bigg\|\frac{1}{h}D^3\tilde{W}(\nabla_h u_h)[\nabla_h\partial_{z_j} u_h]  \nabla_h w\Big)\bigg\|_{L^2(0,L;H^1(S))}\\
		& \leq CR \|\nabla_h w\|_{L^2(\Omega)} + CRh \|\nabla_h^2 w\|_{L^2(\Omega)}.
	\end{align*}
	Hence
	\begin{align}\nonumber
			\bigg\|\bigg(\nabla\frac{1}{h}\e_h (\tw_j), \nabla_h^2 \tw_j \bigg)\bigg\|_{L^2(\Omega)} &\leq Ch^2\Big(\|\tilde{f}_j\|_{L^2(\Omega)} + \|\partial_t^2 \tw_j\|_{L^2(\Omega)}\Big) + C \bigg\|\frac{1}{h}\e_h(\tw_j)\bigg\|_{H^{0,1}(\Omega)}\\
			& \quad  + CR\frac{1}{h}\bigg|\int_\Omega \tw_j \cdot x^\perp dx\bigg| + CR \Big\|\Big(\nabla_h w,\nabla_h^2 w\Big)\Big\|_{L^2(\Omega)}
		\label{LocalStationaryThirdOrderInequality}
	\end{align}	
	for almost every $t\in (0,T)$. With this we can follow a similar argument as in the proof of Theorem \ref{TheoremSecondOrderInequality}. For this we differentiate the equation for $\tw_j$ in time and test with $\partial_t^2 \tw_j$. Then it follows
	\begin{align*}
		(\partial_t^3 \tw_j,& \partial_t^2\tw_j)_{L^2(\Omega)} + \frac{1}{h^2} \Big(D^2\tilde{W}(\nabla_h u_h)\nabla_h \partial_t\tw_j, \nabla_h\partial_t^2\tw_j\Big)_{L^2(\Omega)}\\
		& = (\partial_t \tilde{f}_j, \partial_t^2 \tw_j)_{L^2(\Omega)} - \frac{1}{h^2}\Big(\partial_t\big(\partial_{z_j}D^2\tilde{W}(\nabla_h u_h)\nabla_h w\big), \nabla_h\partial_t^2 \tw_j\Big)_{L^2(\Omega)}\\
		& \qquad- \frac{1}{h^2}\Big(\partial_t D^2\tilde{W}(\nabla_h u_h)\nabla_h \tw_j, \nabla_h\partial^2_t \tw_j \Big)_{L^2(\Omega)}.
	\end{align*}
	because all boundary integrals disappear due to periodicity of $\tw_j$ and $u_h$, respectively, and the Neumann boundary conditions. With this we can follow a similar strategy as in Theorem \ref{TheoremSecondOrderInequality} to obtain an analogous result. Finally using 
	\begin{equation*}
		|A_k| = \bigg| \frac{1}{h} \Big(D^2\tilde{W}(\nabla_h u_{0,h}) \nabla_h w_k,\nabla_h w_k \Big)_{L^2}\bigg| \leq C \bigg\|\frac{1}{h} \e_h(w_k)\bigg\|^2_{L^2} + C\bigg|\frac{1}{h} \int_\Omega w_k \cdot x^\perp dx\bigg|^2
	\end{equation*}
	for $k=0, 1, 2$ we obtain the claimed inequality.	
\end{proof}

\subsection{Proof of Theorem \ref{MainTheorem}}
\label{subsec::ProofOfMainTheorem}
Before we start the proof of the main theorem, we will prepare some bounds on the rotation of the solution around the $x_1$-axis. More precisly we need to bound the following quantity
\begin{equation*}
	\operatorname{max}_{\sigma = 1,2,3} \bigg\|\frac{1}{h} \int_\Omega \partial_t^\sigma u_h\cdot x^\perp dx\bigg\|_{C^0([0,T(h)])}.
\end{equation*}
We can transform the system \eqref{NLS1}--\eqref{NLS4} via $\phi_h \colon \Omega_h \to \Omega$, $x\mapsto (x_1, \frac{1}{h}x_2, \frac{1}{h}x_3)$. Hence $y_h := u_h \circ \phi_h$ solve the equation
\begin{align*}
	\partial_t^2 y_h - \frac{1}{h^2}\operatorname{div} D\tilde{W}(\nabla y_h) &= \hat{f}_h := h^{1+\theta} f_h\circ \phi_h \quad \text{in } \Omega_h \times [0, T_*)\\
	D\tilde{W}(\nabla y_h)\nu|_{(0, L)\times h \partial S} &= 0\\
	y_h \text{ is $L$-periodic} &\text{ w.r.t. } x_1 \\
	(y_h, \partial_t y_h)|_{t=0} &= (y_{0,h}, y_{1,h})
\end{align*}
with $(y_{0,h}, y_{1,h}):=(u_{0,h}\circ\phi_h, u_{1,h}\circ\phi_h)$. Due to the frame invariance the Piola-Kirchhoff stress $DW(F)$ fulfils the appropriate symmetry condition and one can apply the balance law of angular momentum 
\begin{align*}
	\int_{\partial\Omega_h} (x + u_h\circ\phi_h) \times \frac{1}{h^2} D\tilde{W}(\nabla(u_h\circ \phi_h))\nu d\sigma(x)& + \int_{\Omega_h} (x+ u_h\circ\phi_h)\times\hat{f} dx\\
	& = \int_{\Omega_h} (x + u_h\circ\phi_h) \times \partial_t^2 u_h\circ\phi_h dx
\end{align*}
With the transformation formula applied for $\phi^{-1}$ we conclude with $g^h := h^{1+\theta} f^h$
\begin{align*}
	h \int_{\partial\Omega} (\phi^{-1}_h(x) + u_h) \times \frac{1}{h^2} D\tilde{W}(\nabla_h u_h)\nu d\sigma(x) &+ h^2 \int_\Omega (\phi^{-1}_h(x) + u_h) \times g^h dx\\
	& = h^2\int_{\Omega} (\phi^{-1}_h(x) + u_h) \times \partial_t^2 u_h dx
\end{align*}
We can restrict to just the first component, as only rotations around $x_1$-axis have to be controlled for the use of Korn's inequality. For the first component we have
\begin{align*}
	\bigg(\int_{\partial\Omega}& (\phi^{-1}_h(x) + u_h)\times \frac{1}{h^2} D\tilde{W}(\nabla_h u_h)\nu d\sigma(x)\bigg)_1\\
	& = \int_{\partial\Omega} (hx^\perp + u_h^\perp)\cdot \frac{1}{h^2} D\tilde{W}(\nabla_h u_h)\nu d\sigma(x) = \int_{S} \Big[(hx^\perp + u_h) \cdot \frac{1}{h^2} D\tilde{W}(\nabla_h u_h)\nu\Big]_0^L dx'
\end{align*}
because $D\tilde{W}(\nabla_h u_h)\nu = 0$ on $(0,L)\times\partial S$ and as $x^\perp$ does not depend on $x_1$ it follows that $hx^\perp + u_h$ is $L$-periodic in the $x_1$-direction. Using this and $\nu(0,x') = - \nu(L,x')$ for all $x'\in S$ we deduce
\begin{equation*}
	\bigg(\int_{\partial\Omega} (\phi^{-1}_h(x) + u_h)\times \frac{1}{h^2} D\tilde{W}(\nabla_h u_h)\nu d\sigma(x)\bigg)_1 = 0.
\end{equation*}
Thus we have
\begin{equation}
\partial_t^2\int_\Omega x^\perp\cdot u_h dx = \int_\Omega x^\perp\cdot g^h dx + \frac{1}{h}\int_\Omega u_h^\perp\cdot g^h dx - \frac{1}{h}\int_\Omega u_h^\perp\cdot\partial_t^2 u_hdx
\label{MomentumBalanceLaw}
\end{equation}
for almost all $t\in [0,T_{max}(h))$. With this we can later bound $\|\frac{1}{h} \int_\Omega \partial_t^\sigma u_h\cdot x^\perp dx\|_{C^0([0, T(h)])}$ for $\sigma = 1,2,3$ uniformly in $h\in (0,1]$. We note that with \eqref{MomentumBalanceLaw} it follows
\begin{align}
	&\bigg\|\frac{1}{h} (\partial_t u_h, x^\perp)_{L^2(\Omega)}\bigg\|_{C^0([0,T(h)])} \leq \bigg|\frac{1}{h} (u_{1,h}, x^\perp)_{L^2(\Omega)}\bigg| + \frac{1}{h} \int_0^{T(h)} |(x^\perp, g^h)_{L^2}| d\tau \label{InequalityFirstOrderRoation}\\
	& \qquad + \frac{1}{h^2} \int_0^{T(h)} |(u_h^\perp, g^h)_{L^2}| d\tau + \frac{1}{h^2} \int_0^{T(h)} |(u_h^\perp, \partial_t^2 u_h)_{L^2}|d\tau\notag\\
	& \bigg\|\frac{1}{h} (\partial_t^2 u_h, x^\perp)_{L^2(\Omega)}\bigg\|_{C^0([0,T(h)])} \leq \bigg\|\frac{1}{h} (g^h, x^\perp)_{L^2(\Omega)}\bigg\|_{C^0([0,T(h)])}\label{InequalitySecondOrderRoation}\\
	&\qquad + \bigg\|\frac{1}{h^2} (u_h^\perp, g^h)_{L^2(\Omega)}\bigg\|_{C^0([0,T(h)])} +\bigg\|\frac{1}{h^2} (u_h^\perp, \partial_t^2 u_h)_{L^2(\Omega)}\bigg\|_{C^0([0,T(h)])} \notag\\
	&\bigg\|\frac{1}{h} (\partial_t^3 u_h, x^\perp)_{L^2(\Omega)}\bigg\|_{C^0([0,T(h)])} \leq \bigg\|\frac{1}{h} (x^\perp, \partial_t g^h)_{L^2(\Omega)}\bigg\|_{C^0([0,T(h)])}\label{InequalityThirdOrderRoation}\\
	& \qquad + \bigg\|\frac{1}{h^2} (\partial_t u_h^\perp, g^h)_{L^2(\Omega)}\bigg\|_{C^0([0,T(h)])} + \bigg\|\frac{1}{h^2} (u_h^\perp, \partial_t g^h)\bigg\|_{C^0([0,T(h)])} \notag\\
	&\qquad + \bigg\|\frac{1}{h^2} (\partial_t u_h^\perp, \partial_t^2 u_h)_{L^2(\Omega)} \bigg\|_{C^0([0,T(h)])} + \bigg\|\frac{1}{h^2} (u_h^\perp, \partial_t^3 u_h)_{L^2(\Omega)} \bigg\|_{C^0([0,T(h)])} \notag
\end{align}
\begin{proof}[Proof of Theorem \ref{MainTheorem}]
	Without loss of generality we will assume that $0<T\leq 1$. This is possible as we can perform a rescaling in $h$ and $t$ by $T^{-1}$, changing only $M$ by a $T$-depending factor. Furthermore we assume that $R_0$ is sufficiently small, such that all results of Section \ref{subsec::UniformEstimatesForLinearisedSystem} are applicable.
	
	The assumptions \eqref{AssumptionOnInitialDataU1}--\eqref{AssumptionOnInitialDataU3} and \eqref{AssumptionsOnNonlinearf} imply
	\begin{align}
		&\max_{|\alpha| = 1} h^{1+\theta}\Big(\|\partial_t^2 \partial_z^\alpha f\|_{L^1(0,T;L^2)} + \|\partial_t \partial_z^\alpha f\|_{L^\infty(0,T;L^2)\cap L^1(0,T;H^{0,1})} + \|\partial_z^\alpha f\|_{L^\infty(0,T;H^1)}\Big)\notag\\
		&\;+ \max_{k=0,1,2} \bigg(\Big\|\Big(\frac{1}{h} \e_h(u_{1+k,h}), \partial_{x_1} \frac{1}{h} \e_h(u_{k,h})\Big)\Big\|_{L^2} + \bigg|\frac{1}{h} \int_\Omega u_{1+k, h} \cdot x^\perp dx\bigg|\bigg)\notag\\
		&\;+ \max_{|\alpha| = 1} h^{1+\theta} \|\partial_z^\alpha f|_{t=0}\|_{L^2} + \Big\|\Big(\frac{1}{h}\e_h(\partial_{x_1} u_{2,h}), \frac{1}{h}\e_h(\partial^2_{x_1} u_{1,h}), \partial_{x_1} u_{3,h}, \partial^2_{x_1} u_{2,h}\Big)\Big\|^2_{L^2}\notag\\
		&\;+ \Big\|u_{4,h} - \frac{1}{h^2} \divh\Big(D^2\tilde{W}(\nabla_h u_{0,h}) [\nabla_h u_{1,h}, \nabla_h u_{1,h}]\Big)\Big\|_{L^2}\notag\\
		&\;+\Big\|\partial_{x_1} u_{3,h} - \frac{1}{h^2} \divh\Big(D^2\tilde{W}(\nabla_h u_{0,h}) [\nabla_h u_{1,h}, \nabla_h \partial_{x_1} u_{0,h}]\Big)\Big\|_{L^2}\notag\\
		&\;+ \max_{k=1,2} \|(u_{1+k,h}, \partial_{x_1} u_{k,h})\| \leq \tilde{M} h^{1+\theta} \label{localInequalityForInitialValues}
	\end{align}
	for $\tilde{M} = CM$ with some universal constant $C\geq 1$. We choose $h_0 >0$ small enough such that $R:=106 C_{max}\tilde{M} h_0^\theta \leq R_0$ holds, where $C_{max} \geq 1$ is chosen as in Corollary \ref{UniversalInequality}. Let $u_h\in \bigcap_{k=0}^4 C^k([0,T_{max}(h)); H^{4-k}_{per})$ be the solution of \eqref{NLS1}--\eqref{NLS4} from Theorem \ref{TheoremNonlinearShortTimeExistence}. Then there exists some maximal $0 < T'=T'(h) \leq\min\{T_{max}(h), T\}$ such that
	\begin{align}
	\max_{\substack{|\alpha| \leq 1, |\beta|\leq 2, |\gamma|\leq 1\\\sigma=0,1,2}} \bigg(\Big\|\Big(\partial_t^{2+\sigma}& u_h, \nabla_{x,t}^\beta\frac{1}{h} \e_h(\partial_z^\alpha u_h), \nabla_{x,t}^\gamma \nabla_h^2 \partial_z^\alpha u_h\Big)\Big\|_{C^0([0,T'(h)];L^2(\Omega))}  \label{localInequalityForU}\\
	&\qquad + \bigg\|\frac{1}{h}\int_\Omega \partial_z^{\alpha + \beta} u_h\cdot x^\perp dx\bigg\|_{C^0([0,T'(h)])} \bigg) \leq 106 C_{max} \tilde{M} h^{1+\theta}. \notag
	\end{align}
	This maximum exists since as, if \eqref{localInequalityForU} holds, the set $\{\nabla_h u_h(x,t) \;:\; x\in\overline{\Omega}, t\in[0,T'(h)]\}$ is precompact in $U_h$, cf. Appendix \ref{Appendix::ExistenceOfClassicalSolutionsForFixedH}. Moreover it holds
	\begin{equation}
	\int_0^{T'(h)} \|\nabla_{x,t}^2 u_h\|_{L^\infty(\Omega)} dt < \infty.
	\label{localNoBlowUpCondition}
	\end{equation}
	This can be seen by using \eqref{localInequalityForU}, as it follows 
	\begin{equation*}
		\nabla_h^2 u_h(t,\cdot)\in H^{1,1}(\Omega)\hookrightarrow H^1(0,L; H^1(S))\hookrightarrow\operatorname{BUC}([0,L]; L^\infty(S))
	\end{equation*}
	and 
	\begin{equation*}
		\partial_t^2 u(t,\cdot)\in H^2(\Omega)\hookrightarrow C^0(\overline{\Omega})
	\end{equation*}
	for all $t\in [0, T'(h))$. Moreover as long as \eqref{localInequalityForU} is valid, $u_h$ satisfies \eqref{AssumptionsOnU1} and all the results of Section \ref{subsec::UniformEstimatesForLinearisedSystem}, especially Corollary \ref{UniversalInequality}, are applicable. We want to reduce to the case that $u_h$ is mean value free. Hence we assume in a first step that 
	\begin{equation*}
		\int_\Omega u_h dx = \int_\Omega u_{0,h} dx =\int_\Omega u_{1,h} dx =\int_\Omega f^h dx = 0
	\end{equation*} 
	for all $t\in [0,T'(h)]$. Using \eqref{NLS1}--\eqref{NLS4}, we obtain that $w_h^j := \partial_{z_j} u_h$, $j = 0, 1$, solves 
	\begin{gather*}
	\partial_t^2 w_h^j - \frac{1}{h^2} \divh(D^2\tilde{W}(\nabla_h u_h)\nabla_h w_h^j) =h^{1+\theta} \partial_{z_j} f_h \quad\text{in } \Omega\times (0,T'(h))\\
	D^2\tilde{W}(\nabla_h u_h)[\nabla_h w_h^j] \nu|_{(0,L)\times\partial S} = 0 \\
	w_h^j \text{ is $L$-periodic in $x_1$}\\
	(w_h^j, \partial_t w_h^j)\Big|_{t=0} = (w_{0,h}^j, w_{1,h}^j)
	\end{gather*}
	with $w_{k,h}^0 = u_{1+k,h}$ and $w_{k,h}^1 = \partial_{x_1} u_{k,h}$. Hence with Theorem \ref{UniversalInequality} and \eqref{localInequalityForInitialValues} it follows
	\begin{align*}
	&\max_{\substack{|\alpha| = 1, |\beta|\leq 2, |\gamma|\leq 1\\\sigma=0,1,2}} \bigg(\bigg\|\bigg(\partial_t^{2+\sigma} u_h, \nabla_{x,t}^\beta\frac{1}{h}\e_h(\partial_z^\alpha u_h), \nabla_{x,t}^\gamma\nabla_h^2 \partial_z^\alpha u_h\bigg)\bigg\|_{C^0([0,T'(h)], L^2)} \\
	& \hspace{2cm}+ \bigg\|\frac{1}{h} \int_\Omega \partial_z^{\alpha + \beta} u_h \cdot x^\perp dx\bigg\|_{C^0([0,T'(h)])}\bigg)\\
	& \leq C_{max} e^{C \tilde{M}h_0^\theta}\bigg(\tilde{M}h^{1+\theta} + \frac{1}{h}\bigg\|\int_\Omega \partial_t^3 u_h \cdot x^\perp  dx\bigg\|_{L^\infty(0,T'(h))}\\
	&\quad  + 2R \max_{\sigma = 0,1,2} \bigg\|\frac{1}{h} \int_\Omega\partial_t^{1+\sigma} u_h \cdot x^\perp dx\bigg\|_{L^\infty(0,T'(h))}\bigg) + \max_{\sigma = 0,1,2} \bigg\|\frac{1}{h} \int_\Omega \partial_t^{1 + \sigma} u_h \cdot x^\perp dx\bigg\|_{C^0([0,T'(h)])}\\
	& \leq C_{max} e^{C \tilde{M}h_0^\theta}\bigg(\tilde{M}h^{1+\theta} + 4 \max_{\sigma = 0,1,2} \bigg\|\frac{1}{h} \int_\Omega \partial_t^{1 + \sigma} u_h \cdot x^\perp dx\bigg\|_{C^0([0,T'(h)])}\bigg),
	\end{align*}
	where we used  $(\partial_{x_1} w, x^\perp)_{L^2(\Omega)} = 0$ and note that $\partial_t^2 \partial_{x_1}^2 u_h$ can be estimated by
	\begin{equation*}
	\|\partial_t^2 \partial_{x_1}^2 u_h\|_{L^2(\Omega)} \leq \|\nabla_h \partial_t^2 \partial_{x_1} u_h\|_{L^2(\Omega)} \leq C \bigg\|\nabla \frac{1}{h}\e_h(\partial_t^2 u_h)\bigg\|_{L^2(\Omega)}
	\end{equation*}
	due to Korn's inequality. 
	Now we want to apply \eqref{InequalityFirstOrderRoation}--\eqref{InequalityThirdOrderRoation} in order to bound the rotational part of $u_h$. It follows for $\sigma\in\{0,1,2\}$
	\begin{align*}
		&\bigg\|\frac{1}{h}  (\partial_t^{1+\sigma} u_h, x^\perp)_{L^2}\bigg\|_{C^0([0,T'(h)])} \leq \bigg\|\frac{1}{h} (u_{1,h}, x^\perp)_{L^2}\bigg\|_{C^0([0,T'(h)])} + \bigg\|\frac{1}{h^2} (\partial_t u_h^\perp, g^h)_{L^2}\bigg\|_{C^0([0,T'(h)])}\\
		& \qquad  + \bigg\|\frac{1}{h} (\partial_t u_h^\perp, \partial_t^2 u_h)_{L^2}\bigg\|_{C^0([0,T'(h)])} + \max_{k=0,1} \bigg(\bigg\|\frac{1}{h} (x^\perp, \partial_t^k g^h)_{L^2}\bigg\|_{C^0([0,T'(h)])} \\
		& \qquad + \bigg\|\frac{1}{h^2} (u_h^\perp, \partial_t^k g^h)_{L^2}\bigg\|_{C^0([0,T'(h)])} + \bigg\|\frac{1}{h^2} (u_h^\perp, \partial_t^{2+k} u_h)_{L^2}\bigg\|_{C^0([0,T'(h)])}\bigg)	
	\end{align*}
	Due to the assumptions on initial values $u_{1,h}$ and the external force $f^h$, we obtain that 
	\begin{equation*}
		\bigg|\frac{1}{h} (u_{1, h}, x^\perp)_{L^2(\Omega)}\bigg| \leq \tilde{M} h^{1+\theta}
	\end{equation*}
	and 
	\begin{equation*}
		\max_{\sigma=0,1} \bigg\|\frac{1}{h} (x^\perp, \partial_t^\sigma g^h)_{L^2}\bigg\|_{C^0([0,T'(h)])} \leq \tilde{M} h^{1+\theta}.
	\end{equation*}
	Moreover for $\sigma \in \{0,1\}$ we deduce with the Cauchy-Schwarz and Poincaré inequality 
	\begin{align*}
		\bigg\|\frac{1}{h^2} (u_h^\perp, \partial_t^\sigma g^h)_{L^2}\bigg\|_{C^0([0,T'(h)])} & \leq \frac{1}{h^2} \|u_h\|_{C^0([0,T'(h)];L^2)} \|\partial_t^\sigma g^h\|_{C^0([0,T'(h)]; L^2)}\\
		& \leq \frac{C_p}{h^2} \|\nabla_h u_h\|_{C^0([0,T'(h)];L^2)} \|\partial_t^\sigma g^h\|_{C^0([0,T'(h)];L^2)} \leq C_p \tilde{M} h^{1+\theta}
	\end{align*}
	as $h^{\theta - 1} \leq 1$ for $h\in (0,1]$. Similarly we obtain
	\begin{equation*}
		\bigg\|\frac{1}{h^2} (\partial_t u_h^\perp, \partial_t^\sigma g^h)_{L^2}\bigg\|_{C^0([0,T'(h)])} \leq C_p \tilde{M} h^{1+\theta}
	\end{equation*}
	and
	\begin{align*}
		\bigg\|\frac{1}{h^2} (u_h^\perp, \partial_t^{2+\sigma} u_h)_{L^2}\bigg\|_{C^0([0,T'(h)])} &\leq \frac{C_p^2}{h^2} \|\nabla_h u_h\|_{C^0([0,T'(h)]; L^2)} \|\nabla_h\partial_t^{2+\sigma} u_h\|_{C^0([0,T'(h)];L^2)}\\
		& \leq C_p^2 \tilde{M} h^{1+\theta}.
	\end{align*}
	Altogether this leads to
	\begin{equation*}
		\max_{\sigma = 0,1,2} \bigg\|\frac{1}{h}  (\partial_t^{1+\sigma} u_h, x^\perp)_{L^2}\bigg\|_{C^0([0,T'(h)])} \leq 6\tilde{M}h^{1+\theta}.
	\end{equation*}
	Thus we have shown
	\begin{align}
		\max_{\substack{|\alpha| = 1, |\beta|\leq 2, |\gamma|\leq 1\\\sigma=0,1,2}}& \bigg(\bigg\|\bigg(\partial_t^2\partial_t^\sigma u_h, \nabla_{x,t}^\beta\frac{1}{h}\e_h(\partial_z^\alpha u_h), \nabla_{x,t}^\gamma\nabla_h^2 \partial_z^\alpha u_h\bigg)\bigg\|_{C^0([0,T'(h)];L^2)} \label{localInequalityHigherDerivatives} \\
		& \hspace{1cm}+ \bigg\|\frac{1}{h} \int_\Omega \partial_t^{1 + \sigma} u_h \cdot x^\perp dx\bigg\|_{C^0([0,T'(h)])}\bigg)\leq 25 C_{max} e^{C \tilde{M}h_0^\theta} \tilde{M}h^{1+\theta}.\notag
	\end{align}
	Moreover, exploiting 
	\begin{equation*}
		u_h = u_{0,h} + \int_0^t \partial_t u_h ds\quad\text{ and }\quad \nabla_h u_h = \nabla_h u_{0,h} + \int_0^t \nabla_h w_h^0 ds
	\end{equation*}
	we obtain with \eqref{AssumptionOnInitialDataU2} 
	\begin{align*}
	\|\nabla_h^2 u_h\|_{C^0([0,T'(h)]; L^2)} &\leq \|\nabla_h^2 u_{0,h}\|_{L^2} + T \|\nabla_h^2 \partial_t u_h\|_{C^0([0,T'(h)]; L^2)}\\
	& \leq Mh^{1+\theta} + 25 C_{max} e^{C \tilde{M}h_0^\theta} \tilde{M}h^{1+\theta}
	\end{align*} 
	and 
	\begin{align*}
	&\bigg\|\frac{1}{h} \e_h(u_h)\bigg\|_{C^0([0,T'(h)], L^2)} + \bigg\|\frac{1}{h} \int_\Omega u_h\cdot x^\perp dx\bigg\|_{C^0([0,T'(h)])}\\
          &\quad \leq \bigg\|\frac{1}{h} \e_h(u_{0,h})\bigg\|_{L^2} + T\bigg\|\frac{1}{h} \e_h(\partial_t u_h)\bigg\|_{C^0([0,T'(h)], L^2)}  + \bigg|\frac{1}{h} \int_\Omega u_{0,h}\cdot x^\perp dx\bigg|\\
          &\qquad + T \bigg\|\frac{1}{h} \int_\Omega \partial_t u_{0,h}\cdot x^\perp dx\bigg\|_{C^0([0,T'(h)])} \leq   2Mh^{1+\theta} + 50 C_{max} e^{C \tilde{M}h_0^\theta} \tilde{M}h^{1+\theta}
	\end{align*}
	due to \eqref{AssumptionOnInitialDataU1}, \eqref{AssumptionOnInitialDataU3} and \eqref{localInequalityHigherDerivatives}. Hence we deduce
	\begin{align}
	&\max_{\substack{|\alpha| \leq 1, |\beta|\leq 2, |\gamma|\leq 1\\\sigma=0,1,2}} \bigg(\bigg\|\bigg(\partial_t^2\partial_t^\sigma u_h, \nabla_{x,t}^\beta\frac{1}{h}\e_h(\partial_z^\alpha u_h), \nabla_{x,t}^\gamma\nabla_h^2 \partial_z^\alpha u_h\bigg)\bigg\|_{C^0([0,T'(h)]; L^2)}\notag \\
	&\hspace{3cm}+ \bigg\|\frac{1}{h} \int_\Omega \partial_z^{\alpha + \beta} u_h \cdot x^\perp dx\bigg\|_{C^0([0,T'(h)])}\bigg) \leq 103 C_{max} e^{C \tilde{M}h_0^\theta} \tilde{M}h^{1+\theta}. \label{localInequaltiyFinalZeroMeanValue}
	\end{align}
	As $\theta > 0$ we can find $h_0>0$ such that 
	\begin{align*}
		&\max_{\substack{|\alpha| \leq 1, |\beta|\leq 2, |\gamma|\leq 1\\\sigma=0,1,2}} \bigg(\bigg\|\bigg(\partial_t^2\partial_t^\sigma u_h, \nabla_{x,t}^\beta\frac{1}{h}\e_h(\partial_z^\alpha u_h), \nabla_{x,t}^\gamma\nabla_h^2 \partial_z^\alpha u_h\bigg)\bigg\|_{C^0([0,T'(h)]; L^2)}\\
		&\hspace{3cm}+ \bigg\|\frac{1}{h} \int_\Omega \partial_z^{\alpha + \beta} u_h \cdot x^\perp dx\bigg\|_{C^0([0,T'(h)])}\bigg) \leq 104 C_{max}\tilde{M}h^{1+\theta}. 
	\end{align*}
	uniformly in $h\in(0,h_0]$. 
	
	Now we have to consider the case that the force or the initial data is not mean value free. In this case we define
	\begin{equation}
		a(t) := \frac1{|\Omega|}\left(\int_\Omega u_{0,h}(t) dx - t\int_\Omega u_{1,h} (t) - \int_0^t (t-s) \int_\Omega f^h(s)dx ds\right).
		\label{GlobalReductionToMeanValueFreeDataDefinitionOfA}
	\end{equation}
	Then $\tilde{u}_h(t) := u_h(t) - a(t)$ solves
	\begin{align*}
		\partial_t^2 \tilde{u}_h - \frac{1}{h^2}\divh D\tilde{W}(\nabla_h \tilde{u}_h) &= h^{1+\theta} \tilde{f}_h \quad \text{in } \Omega\times [0, T),\\
		D\tilde{W}(\nabla_h \tilde{u}_h)\nu|_{(0, L)\times \partial S} &= 0,\\
		\tilde{u}_h \text{ is $L$-periodic} &\text{ w.r.t. } x_1, \\
		(\tilde{u}_h, \partial_t \tilde{u}_h)|_{t=0} &= (\tilde{u}_{0,h}, \tilde{u}_{1,h}),
	\end{align*}
	where we subtracted from $(f, u_{0,h}, u_{1,h})$ their mean values to obtain $(\tilde{f}, \tilde{u}_{0,h}, \tilde{u}_{1,h})$. 
	
	Then it holds for $\tilde{u}$
	\begin{equation*}
		\int_\Omega \tilde{u}_h(t) dx = \int_\Omega (u_h(t) - a(t))dx = 0
	\end{equation*}
	as, integration of the nonlinear equation \eqref{NLS1} implies with the boundary and periodicity condition, \eqref{NLS2} and \eqref{NLS3}, respectively
	\begin{equation*}
		\partial_t^2 \int_\Omega u(x,t) dx = \int_\Omega f^h(x,t) dx.
	\end{equation*}
	Moreover \eqref{AssumptionOnInitialDataU3} and \eqref{AssumptionsOnRotationOfNonlinearf} is fulfilled for $\tilde{u}_{k,h}$, $k\in\{0,1,2,3\}$ and $\tilde{f}$, because $\int_\Omega x^\perp dx = 0$. Deploying the fact that the initial data is only changed by a constant, \eqref{AssumptionOnInitialDataU2} holds for the new initial values. With $L^2(\Omega) \hookrightarrow L^1(\Omega)$ and triangle inequality it follows \eqref{AssumptionsOnNonlinearf} with $\tilde{C}M$ instead of $M$, for some $\tilde{C}\geq 1$ independent of $h_0$, $h$ and $M$. In the same way one can deal with $\tilde{u}_{2+k,h}$ in \eqref{AssumptionsOnU1}. Hence, as for \eqref{AssumptionOnInitialDataU2}, we obtain that \eqref{AssumptionOnInitialDataU1} holds with $M$ replaced by $\tilde{C}M$. Thus we can apply \eqref{localInequaltiyFinalZeroMeanValue} for $\tilde{u}_h$ and get
	\begin{align*}
		&\max_{\substack{|\alpha| \leq 1, |\beta|\leq 2, |\gamma|\leq 1\\\sigma=0,1,2}} \bigg(\bigg\|\bigg(\partial_t^2\partial_t^\sigma \tilde{u}_h, \nabla_{x,t}^\beta\frac{1}{h}\e_h(\partial_z^\alpha \tilde{u}_h), \nabla_{x,t}^\gamma\nabla_h^2 \partial_z^\alpha \tilde{u}_h\bigg)\bigg\|_{C^0([0,T'(h)]; L^2)}\\
		&\hspace{3cm}+ \bigg\|\frac{1}{h} \int_\Omega \partial_z^{\alpha + \beta} \tilde{u}_h \cdot x^\perp dx\bigg\|_{C^0([0,T'(h)])}\bigg) \leq 103 C_{max}\tilde{C} e^{C \tilde{C}\tilde{M}h_0^\theta} \tilde{M}h^{1+\theta}.
	\end{align*}
	From the definition of $\tilde{u}_h$ it follows 
	\begin{equation*}
		\nabla_{x,t}^\beta\frac{1}{h}\e_h(\partial_z^\alpha \tilde{u}_h) = \nabla_{x,t}^\beta\frac{1}{h}\e_h(\partial_z^\alpha u_h) 
		\quad \text{ and }\quad \nabla_{x,t}^\gamma\nabla_h^2 \partial_z^\alpha \tilde{u}_h = \nabla_{x,t}^\gamma\nabla_h^2 \partial_z^\alpha u_h.
	\end{equation*}
	Because of $\int_\Omega x^\perp dx = 0$, we have
	\begin{equation*}
		\frac{1}{h} \int_\Omega \partial_z^{\alpha + \beta} \tilde{u}_h \cdot x^\perp dx = \frac{1}{h} \int_\Omega \partial_z^{\alpha + \beta} u_h \cdot x^\perp dx.
	\end{equation*}
	Lastly for $\sigma\in\{0,1,2\}$ we deduce from $\partial_t^{2+\sigma} a(t) = h^{1+\theta} \int_\Omega \partial_t^{\sigma} f^h dx$ and \eqref{AssumptionsOnRotationOfNonlinearf} that
	\begin{align*}
		\max_{\sigma = 0,1,2} &\|\partial_t^{2+\theta} u_h\|_{C^0([0,T'(h)]; L^2)}\\
		&\leq \max_{\sigma = 0,1,2}\|\partial_t^{2+\sigma} \tilde{u}_h\|_{C^0([0,T'(h)];L^2)} + \max_{\sigma = 0,1,2}\|\partial_t^{2+\sigma} a(t)\|_{C^0([0,T'(h)];L^2)}\\
		& \leq  \max_{\sigma = 0,1,2}\|\partial_t^{2+\sigma} \tilde{u}_h\|_{C^0([0,T'(h)];L^2)} + Mh^{1+\theta}
	\end{align*}
	Thus it follows 
	\begin{align*}
		&\max_{\substack{|\alpha| \leq 1, |\beta|\leq 2, |\gamma|\leq 1\\\sigma=0,1,2}} \bigg(\bigg\|\bigg(\partial_t^2\partial_t^\sigma u_h, \nabla_{x,t}^\beta\frac{1}{h}\e_h(\partial_z^\alpha u_h), \nabla_{x,t}^\gamma\nabla_h^2 \partial_z^\alpha u_h\bigg)\bigg\|_{C^0([0,T'(h)]; L^2)} \\
		&\hspace{3cm}+ \bigg\|\frac{1}{h} \int_\Omega \partial_z^{\alpha + \beta} u_h \cdot x^\perp dx\bigg\|_{C^0([0,T'(h)])}\bigg) \leq 104 C_{max}\tilde{C} e^{C \tilde{C}\tilde{M}h_0^\theta} \tilde{M}h^{1+\theta}\\
		&\qquad\leq 105 C_{max} \tilde{M} h^{1+\theta}
	\end{align*}
	for $h_0$ small.
\end{proof}
\appendix
\section{Appendix: Existence of classical solutions for fixed $h>0$}
\label{Appendix::ExistenceOfClassicalSolutionsForFixedH}
In this appendix we want to give more details on how the existence result of \cite{Koch} is applied in the regarded situation. First we shortly summaries the assumptions and equation considered in \cite{Koch} and the main result \cite[Theorem 1.1]{Koch}, which we want to apply. Second we give some remarks on how our system is obtained and why the assumptions assumed in Theorem \ref{MainTheorem} are sufficient.

In \cite{Koch} a quasi-linear hyperbolic equation of the form 
\begin{align}
\sum_{i=0}^{n} \partial_{x_i} F^i_j(t,x,u,Du) &= w_j(t,x,u,Du) &&\hspace{-1cm}\text{in } \Omega\times(0,T) \label{globalSystemKoch1}\\
\sum_{i=0}^{n} \nu_i F^i_j(t,x,u,Du) &= g_j &&\hspace{-1cm}\text{on } \partial\Omega\times(0,T) \label{globalSystemKoch2}\\
(u|_{t=0},\partial_t u|_{t=0}) &= (u_0, u_1)&&\hspace{-1cm}\text{in } \Omega \label{globalSystemKoch3}
\end{align}
is considered, where $1 \leq j \leq N$, $x_0 = t$, $\Omega\subset\R^n$ is a bounded domain with boundary of class $C^{s+2}$, $s>n/2 + 1$ and $\nu$ is the outer normal. Moreover $u\colon \Omega\times [0,T)\to\R^n$ and $Du = (\partial_t u, \partial_{x_1} u, \ldots, \partial_{x_n} u)$, with $0 < T \leq \infty$. For convenience we state the assumptions made in \cite{Koch} in a slightly simplified version. 
\begin{itemize}
	\myitem[A1]\label{itm:A1} {We assume that $u_0\in H^{s+1}(\Omega)$, $u_1\in H^s(\Omega)$ and let $U$ be an open neighbourhood of $\{0\}\times\operatorname{graph}((u_0, u_1, D_x u_0))$ in $[0,T)\times\overline{\Omega}\times\R^N\times\R^N\times\R^{N\times n}$. Moreover assume $F$, $g\in C^{s+1}(U)$, $w\in C^s(U)$ and define
		\begin{equation*}	
		a^{ik}_{jl} = \frac{\partial F^i_j}{\partial(\partial_{x_k} u^l)}(t, x, u, Du) \text{ for all } 0\leq i,k\leq n, 1 \leq l,k\leq N.
		\end{equation*}}
	\myitem[A2]\label{itm:A2} {$a^{ik}_{jl} = a^{ki}_{lj}$ in $U$ for all $0\leq i, k\leq n$, $1 \leq l,k\leq N$.}
	\myitem[A3]\label{itm:A3} {For any $t\in [0,T)$, $v_0\in C^1(\overline{\Omega})$ and $v_1\in C(\overline{\Omega})$ with $\{t\}\times\operatorname{graph}((v_0, v_1, D_x v_0))\subset U$ there exists $\kappa_0 > 0$ and $\mu \geq 0$ such that for all $\psi\in C^\infty(\overline{\Omega})$ the inequality
		\begin{equation*}
		\sum_{\alpha,\beta=1}^{n} \sum_{j,l=1}^{N} (a^{\alpha\beta}_{jl} \partial_{x_\beta}\psi_j, \partial_{x_\alpha}\psi_l)_{L^2(\Omega)} \geq \kappa_0 \|\psi\|^2_{H^1(\Omega)} - \mu \|\psi\|^2_{L^2(\Omega)}.
		\end{equation*}}
	\myitem[A4]\label{itm:A4} {For any $\xi\in U$ there exists a $\kappa > 0$ such that for all $\eta\in\R^N$ the inequality
		\begin{equation*}
		a^{00}_{jl} (\xi)\eta_j \eta_l \leq -\kappa_1 |\eta|^2
		\end{equation*}
		holds.}
	\myitem[A5]\label{itm:A5} {We suppose that the compatibility condition holds up to order $s$.}
\end{itemize}
Under these assumptions the following theorem holds:
\begin{theorem}[Theorem 1.1, \cite{Koch}]~\\
	There exists a unique $0 < t_0 \leq T$ and a unique classical solution $u\in C^2([0,t_0)\times \overline{\Omega})$ of \eqref{globalSystemKoch1}--\eqref{globalSystemKoch3} such that $D^\sigma u(t)\in L^2(\Omega)$ for $0\leq \sigma\leq s+1$. Moreover $t_0$ is characterised by the two alternatives: either the graph of $(u, Du)$ is not precompact in $U$ or 
	\begin{equation*}
	\int_0^t \|D^2 u(\tau)\|_{L^\infty(\Omega)} d\tau \to \infty\quad \text{ for } t\to t_0.
	\end{equation*}	
\end{theorem}
In the situation this paper the considered domain $\Omega$ is not sufficiently smooth, but due to the periodic boundary condition on the end faces of $\Omega$ the equations \eqref{NLS1}--\eqref{NLS4} are equivalent to the equations on the manifold $M := (\R/L\mathbb{Z}) \times S$. This is a bounded manifold with smooth boundary, as $S$ is a $C^\infty$ domain. The ideas of \cite{Koch} are similar as in Subsection \ref{subsec::UniformEstimatesForLinearisedSystem}, using differentiation in time and applying results from the elliptic theory. 

Choosing $n,N=3$, $g_j \equiv 0$, $w_j (t,x,u,Du) = -h^{1+\theta} (f_h)_j$ and
\begin{equation*}
F^0_j(t,x,u,Du) = -\partial_t u_j, \quad F^i_j (t,x,u,Du) = \frac{1}{h^2}(D\tilde{W}(\nabla_h u))_{j,i} = \frac{1}{h^2} \bigg(\frac{\partial W}{\partial(\partial_{x_i} u_j)}\bigg) (\nabla_h u)
\end{equation*}
for $i,j=1,2,3$. Then we obtain the symmetry condition \ref{itm:A2}. As $(a^{00}_{jl})_{j,l=1,2,3} = - Id\in\Rtimes$ the assumption \ref{itm:A4} is fulfilled, with $\kappa_1 = 1$. Moreover the compatability assumptions of Theorem \ref{MainTheorem} imply \ref{itm:A5}. For the first assumption we choose $s=3$. Then the initial data is sufficiently regular and as $f^h$ does not depend on $(u,Du)$ the prescribed regularity is sufficient. Lastly we can choose $U$ as
\begin{equation*}
U = [0,T)\times\overline{\Omega}\times\R^3\times\R^3\times U_h\quad\text{where } U_h := \bigg\{A\in\Rtimes\;\:\;  \bigg|\bigg(A, \frac{1}{h} \operatorname{sym}(A)\bigg)\bigg| \leq \e h \bigg\} 
\end{equation*}
for some sufficiently small $\e$. This is indeed an applicable neighbourhood as for small $h>0$, it holds $\nabla_h u_{0,h}(x) \in U_h$ for all $x\in\overline{\Omega}$, as $H^2(\Omega) \hookrightarrow C^0(\overline{\Omega})$. Finally due to Lemma \ref{Lemma_DecompD3TildeW}, \eqref{globalCentralCoerciveEstimateForAppendix} and Korn's inequality we obtain that the coerciveness assumption \ref{itm:A3} is satisfied.
\begin{bem}
	We want to give a short remark, why the graph of the solution $u_h$ is precompact in $U_h$ for all $h>0$ as long as \eqref{localInequalityForU} holds. First we note that the neighbourhood $U$ lies in a finite dimensional space. Thus $\mathcal{G}(u,Du) := \{(x,t,u,\partial_t, \nabla u) \;:\; (x,t)\in \overline{\Omega}\times [0,T'(h)]\}$ is precompact if and only if $\{(x,t,u,\partial_t, \nabla u) \;:\; (x,t)\in \overline{\Omega}\times [0,T'(h)]\}$ is bounded and the closure satisfies $\overline{\mathcal{G}(u,Du)} \subset U$. Due to the regularity of $u_h$, it follows that $\mathcal{G}(u_h,Du_h)$ is bounded and for $h_0>0$ sufficiently small we have 
	\begin{equation*}
	\operatorname{dist}(\{\nabla_h u_h(x,t)\;:\; (x,t)\in \overline{\Omega}\times [0,T'(h)]\}, \partial U_h) \geq \e > 0
	\end{equation*}
	for some uniformly $\e>0$. Hence, the graph of $(u_h, Du_h)$ is precompact in $U$.
\end{bem}


\end{document}